\documentclass[12pt]{amsart}
\usepackage[T1]{fontenc}
\usepackage{amsmath, amssymb}
\usepackage[english]{babel}
\usepackage{mathrsfs}
\usepackage[all]{xy}
\usepackage{enumitem}
\usepackage{url}
\usepackage{todonotes}

\newtheorem{thm}{Theorem}[section]
\newtheorem*{thm*}{Theorem}
\newtheorem*{cor*}{Corollary}
\newtheorem{lem}[thm]{Lemma}

\newtheorem{prop}[thm]{Proposition}
\newtheorem*{prop*}{Proposition}

\newtheorem{cor}[thm]{Corollary}

\theoremstyle{definition}
\newtheorem{defn}[thm]{Definition}
\newtheorem{notation}[thm]{Notation}
\newtheorem{remark}[thm]{Remark}

\newtheorem{question}[thm]{Question}

\newtheorem{claim}[thm]{Claim}

\newcommand{\dminus}{ 
\buildrel\textstyle\ .\over{\hbox{ 
\vrule height3pt depth0pt width0pt}{\smash-} 
}}

\def\vp{\varphi}

\def\cal{\mathcal}

\newcommand{\cstar}{$\mathrm{C}^*$}
\newcommand{\wstar}{$\mathrm{W}^*$}

\def \md{\operatorname{mod}}
\def \Md{\operatorname{Mod}}

\DeclareMathOperator{\vNa}{vNa}
\DeclareMathOperator{\spec}{sp}

\newcommand\ip[2]{\left\langle #1\, ,\!\ #2 \right\rangle}
\newcommand\cU{{\cal U}}
\newcommand\cL{{\cal L}}

\newcommand\R{{\cal R}}

\def \sa{\operatorname{sa}}

\def \tb{\operatorname{tb}}

\def \rt{\operatorname{right}}

\def \GR{\operatorname{GR}}
\def \Oc{\operatorname{Oc}}
\def \HA{\operatorname{HA}}

\def \vP{\mathbf{P}}
\def \vQ{\mathbf{Q}}
\def \vR{\mathbf{R}}

\def \vA{\mathbf{A}}

\def \bbC{\mathbb{C}}
\def \bbR{\mathbb{R}}
\def \bbQ{\mathbb{Q}}
\def \bbZ{\mathbb{Z}}
\def \bbN{\mathbb{N}}

\usepackage[bookmarks=true, colorlinks=true, linkcolor=blue, citecolor= red,hypertexnames=false,hyperindex,breaklinks]{hyperref}

\begin{document}


\title[Generalized Ocneanu Ultraproducts]{Model Theory of General von Neumann Algebras I: Generalized Ocneanu Ultraproducts}

\author{Jananan Arulseelan}
\address{Department of Mathematics, Iowa State University, 396 Carver Hall, 411 Morrill Road, Ames, IA 50011, USA}
\email{jananan@iastate.edu}
\urladdr{https://sites.google.com/view/jananan-arulseelan}


\begin{abstract}
This paper collates, presents, and expands upon technology and results obtained as part of the author's PhD thesis.  We generalize work done in the $\sigma$-finite setting by the author, Goldbring, Hart, and Sinclair by producing a language and axiomatization of full left Hilbert algebras.  To improve the accessibility of using this axiomatization, we examine the metric structure ultraproduct associated to this axiomatization. In doing so, we generalize results of Ando-Haagerup and Masuda-Tomatsu.  This examination leads us to multiple operator-algebraic characterizations of the ultraproduct which closely resemble known characterizations of the Ocneanu ultraproduct.  One of these is closely related to the notion of continuous elements of an ultraproduct with respect to an action.  In the spirit of results of the author, Goldbring, and Hart, we prove various undecidable universal theory results for von Neumann algebras with unbounded weights.  Notably, we prove that the hyperfinite II$_\infty$ factor together with its canonical tracial weight has an undecidable theory.  This result was expected by experts, but could not be made precise until now.  We also study the strengthenings of the negative solution to CEP that this result implies.
\end{abstract}

\maketitle

\tableofcontents

\section{Introduction}\label{SectionIntro}

Continuous model theory, and more recently, computable continuous model theory have been used to produce and describe many interesting phenomena and results in the theory of operator algebras.  Restricting our attention to von Neumann algebras, there is a significant body of work on the model theories of tracial von Neumann algebras (see \cite{FG}, \cite{FHS1}, \cite{FHS2}, \cite{FHS3}, \cite{GaoJekel}, \cite{GHHyp}, \cite{GH}, \cite{GH2}) and \wstar-probability spaces (see \cite{AGH}, \cite{AGHS}, \cite{Dab}, \cite{GoldbringHoudayer}).  In both of these situations, a key role is played by the specified faithful normal trace, or faithful normal state, respectively.  This has the drawback of not being applicable to von Neumann algebras that are not $\sigma$-finite or even type II$_\infty$ von Neumann algebras equipped with their natural tracial weights.  For instance, the following result has been expected to be true by experts in continuous logic, but the tools to prove it did not exist until now.

\begin{thm*}
    The universal theory of the hyperfinite II$_\infty$ factor $\R_{0,1}$ with its canonical semifinite trace is undecidable.
\end{thm*}  

The significance of such undecidable universal theory results is demonstrated in \cite{GH2} and \cite{AGH}.  The brief summary is that such results provide vast strengthenings and analogues of the negative solution to the Connes Embedding Problem (which was shown in \cite{MIP}).  In fact, the above theorem has the CEP-esque consequence stated below.  One way to interpret it is that even if we tensor all of the II$_1$ factors in sight by $B(\cal H)$, we still do not get enough wiggle room to make the analogous form of CEP true.

\begin{cor*}
    There is no effectively enumerable theory $T$ extending the theory of $II_{\infty}$ factors such that every model of $T$ embeds into an ultrapower of $\R_{0,1}$.
\end{cor*}

We prove this theorem, and many others, in Section \ref{SectionUniversalTheories}, but first we must make it precise by providing a language in which it makes sense to talk about such theories.  In Section \ref{SectionAxiomatization}, we do this by producing a language $\cal L_{\vNa}$ and a theory $T_{\vNa}$ therein which axiomatize von Neumann algebras with a faithful normal semifinite weight, or equivalently full left Hilbert algebras.  We dub the former ``weighted von Neumann algebras'' for convenience.  Note that since every von Neumann algebra admits a faithful normal semifinite weight, this yields a completely general model theory of von Neumann algebras.

\begin{thm*}
    There is a equivalence of categories between the category of models $\Md(T_{\vNa})$ and the category of weighted von Neumann algebras where embeddings are those which induce a faithful normal conditional expectation onto their image.
\end{thm*}

To accomplish these model-theoretic goals, we need to make use of multiple approaches to Hilbert algebras and their Tomita-Takesaki theory (see \cite{Ha75}, \cite{Kad}, \cite{RvD}, \cite{Takesaki}) and we exposit the relevant theory in Section \ref{SectionHilbertAlg}.  We also need to extend the notion of \textbf{totally bounded elements}, introduced for $\sigma$-finite von Neumann algebras in \cite{AGHS}, to our more general setting.  We do this in Section \ref{SectionTotallyBounded}.

In Section \ref{SectionDefModular}, we furthermore show, by producing a definitional expansion, that there is a computable definition of the modular automorphism group in our theory.

\begin{thm*}
    The modular automorphism group is a definable group acting on models of $T_{\vNa}$.  
\end{thm*}

This, by general model-theoretic considerations, has the following immediate corollary:

\begin{cor*}
    The modular automorphism group commutes with the metric structure ultraproduct in $\Md(T_{\vNa})$.
\end{cor*}

This raises an immediate question: what is this metric structure ultraproduct?  The above result implies that it cannot be the Groh-Raynaud ultraproduct (see \cite{AH}).  Whereas traditionally in continuous theory of operator algebras, our axiomatizations aim to capture a pre-existing ultraproduct, here we are faced with a new ultraproduct and we use model-theoretic techniques to produce operator-algebraic characterizations of it.  This new ultraproduct is something like a hybrid of the Ocneanu and Groh-Raynaud ultraproducts.  

In Section \ref{SectionHAUltraproduct}, we define the \textbf{Hilbert algebra ultraproduct} by naively modifying the traditional definition of the Ocneanu ultraproduct.  We show, using the Arveson spectral theory techniques of \cite{AH}, that this Hilbert algebra ultraproduct agrees with the metric structure ultraproduct for $T_{\vNa}$.  By model-theoretic considerations, we see that the resulting inner product induces a faithful normal semifinite weight on the associated left von Neumann algebra and that the Hilbert algebra ultraproduct commutes with modular automorphisms. 

Thus we are led to our next characterization.  The notion of continuous elements of an ultraproduct with respect to a group action was isolated in \cite{BYG}.  We have a natural embedding of the underlying Hilbert space of the Hilbert algebra ultraproduct into the Banach space ultraproduct of the underlying Hilbert spaces of the factors.  The latter is the representing space of the Groh-Raynaud ultraproduct.  Letting $p_{\Phi}$ denote the projection onto the image of this embedding, we show the following using results about standard forms.

\begin{thm*}
    The left von Neumann algebra $\R_{\lambda}\left( \prod^{\cU}_{\HA} (\cal M_i, \Phi_i) \right)$ of the Hilbert algebra ultraproduct is spatially isomorphic to the corner $p_{\Phi}\left( \prod^{\cU}_{\GR}(\cal M_i, \cal H_{\Phi_i}) \right) p_{\Phi}$ of the Groh-Raynaud ultraproduct.
\end{thm*}

Since the ultraproducts of modular automorphisms are continuous on the image of $p_{\Phi}$, we get the following generalization of a result from \cite{AH} and \cite{MT}, justifying our choice to name it the \textbf{generalized Ocneanu ultraproduct}.

\begin{thm*}
    An element $a \in \prod^{\cU}_{\GR}(\cal M_i, \cal H_{\Phi_i})$ of the Groh-Raynaud ultraproduct represents an element of $\R_{\lambda}\left( \prod^{\cU}_{\HA} (\cal M_i, \Phi_i) \right)$ if and only if $(\sigma^{\Phi_{i}}_{t})^{\bullet}$ is a strong-$^*$ continuous function at $a$.
\end{thm*}

We believe our ultraproduct constructions and their characterizations here are of independent interest to operator algebraists (see, for example, \cite{Cas}).

\section{Background on Hilbert Algebras}\label{SectionHilbertAlg}

The results in \cite{AGHS} rely heavily on the machinery of \cite[Section 4]{RvD}; it is no surprise that our results here rely heavily on \cite[Section 5]{RvD}.  For the reader's convenience, in this section, we recapitulate some of the theory of left Hilbert algebras.  Our treatment combines the bounded operator approach of \cite{RvD} and the more classical approach of \cite[Volume II, Chapter IV]{Takesaki}.

Let $\mathfrak{A}$ be an involutive algebra with involution denoted $\xi \mapsto \xi^\#$.  Assume further that $\mathfrak{A}$ is equipped with an inner product $\ip{\cdot}{\cdot}$ and denote by $\cal H$ the Hilbert space obtained by completing $\mathfrak{A}$ with respect to this inner product.

\begin{defn}\label{DefLeftHilbert}
     We call $\mathfrak{A}$ a \textbf{left Hilbert algebra} if the following conditions all hold.
    \begin{enumerate}
        \item Left multiplication is continuous.  More precisely, for each $a \in \mathfrak{A}$, the linear map $b \mapsto ab$ for $b \in \mathfrak{A}$ extends to a bounded linear operator on $\cal H$.  We denote this extension by $\pi(a)$.
        \item The representation $\pi$ as above is a $^*$-representation.  Namely, we assume $\ip{a b}{c} = \ip{b}{a^\# c}$ for all $a,b,c \in \mathfrak{A}$.
        \item The subalgebra $\mathfrak{A}^2$, the span of $\{ ab \ : \ a,b \in \mathfrak{A}\}$ is dense in $\mathfrak{A}$.  This ensures that $\pi$ is nondegenerate.
        \item the map $a \mapsto a^\#$ and therefore its restriction to $\mathfrak{A}^2$ has closed extension.
    \end{enumerate}
    Taking the weak closure of $\pi(\mathfrak{A})$ give the associated \textbf{left von Neumann algebra} $\R_{\lambda}(\mathfrak{A})$.  We occasionally write $S_0 a := a^\#$.
\end{defn}

We say that $\mathfrak{A}$ is a \textbf{pre-left Hilbert algebra} if it is only assumed to satisfy conditions (1)-(3) above.  We note that this is not standard terminology.  It is introduced here to aid our exposition.  To see why this is a useful definition to make, we direct the reader to the first half of \cite[Section 5]{RvD}.  There, pre-left Hilbert algebras are used extensively but are not given a name.  One can define a \textbf{(pre-)right Hilbert algebra} $\mathfrak{B}$ by analogy, replacing left multiplication by right multiplication whenever the former appears above.  In this case, we  denote the involution by $^\flat$; the induced $^*$-representation by $\pi'$; and the associated \textbf{right  von Neumann algebra} by $\R_{\rho}(\mathfrak{B})$.  We use the following notions extensively.

\begin{defn}
    Let $\mathfrak{A}$ be a pre-left Hilbert algebra with Hilbert space completion $\cal H$.  We call a vector $v \in \cal H$ a \textbf{right bounded vector} if
    \[
    \pi'(v)a := av \qquad \text{for all } a \in \mathfrak{A}
    \]
    extends to a bounded operator, also denoted $\pi'(v)$.  In this case, we say $v$ is \textbf{right $K$-bounded} if $\|\pi'(v)\| \leq K$.  We denote the set of right bounded vectors by $\cal H_{\rt}$.
\end{defn}

\begin{prop}\label{rightBoundedClosed}
    $\cal H_{\rt}$ is closed under:
    \begin{enumerate}
        \item linear combinations, and
        \item $\R_{\lambda}(\mathfrak{A})'$.
    \end{enumerate}
\end{prop}

\begin{proof}
    The proof of $(1)$ is a simple norm calculation.  For $(2)$, let $v \in \cal H_{\rt}$ and $a \in \R_{\lambda}(\mathfrak{A})'$ be given.  Write $w = av$.  Then for all $x \in \mathfrak{A}$, we have
    \[
    \pi'(w)x = xw = xav = axv = a\pi'(v)x 
    \]
    so that $\|\pi'(w)\| \leq \|a\|\|\pi'(v)\|$ and hence $w \in \cal H_{\rt}$.
\end{proof}

\begin{defn}
    Let $\mathfrak{A}'$ denote the set of elements $v \in \cal H$ such that $v$ is right bounded and there exists another vector $v_1 \in \cal H$ such that
    \[
    \pi'(v_1) = \pi'(v)^*.
    \]
    It is clear by the non-degeneracy of $\pi$, that $v_1$ is unique.  We write $v^\flat = v_1$ as above.
\end{defn}

When $\mathfrak{A}$ is a left Hilbert algebra, it can be seen that $\mathfrak{A}'$ can be made into a right Hilbert algebra.  Analogously, given a right Hilbert algebra $\mathfrak{A}'$, we can then define a left Hilbert algebra $\mathfrak{A}{''}$.  A left Hilbert algebra is said to be \textbf{full} if $\mathfrak{A} = \mathfrak{A}{''}$.  Later in this section, we see, and often use, the fact that full left Hilbert algebras are equivalent to von Neumann algebras equipped with a choice of faithful normal semifinite weight.

Let $\mathfrak{A}$ be a pre-left Hilbert algebra with Hilbert space closure $\cal H$.  Let $\cal K$ denote the real subspace of $\cal H$ given by the closed span of elements of the form $a^\# a$, where we consider $\cal H$ as a real Hilbert space with respect to real inner product given by the real part of $\ip{\cdot}{\cdot}$.

\begin{thm}\label{HilbertEquivalentConditions}
    Let $\mathfrak{A}$ be a pre-left Hilbert algebra.  Then the following are equivalent:
    \begin{enumerate}
        \item $\mathfrak{A}$ satisfies condition (4) of Definition \ref{DefLeftHilbert} and, therefore is a left Hilbert algebra;
        \item $\cal K \cap i\cal K = \{0\}$;
        \item $\mathfrak{A}'$ is dense in $\cal H$; and
        \item The set $\cal H_{\rt}$ of right bounded elements is dense in $\cal H$.
    \end{enumerate}
\end{thm}

\begin{proof}
    (1) $\iff$ (2) is the content of the first half of \cite[Section 5]{RvD}.

    Next we show (3) $\implies$ (2).  By the discussion after \cite[Proposition 5.3]{RvD}, we see that $\mathfrak{A}' \subseteq i\cal K^\perp + \cal K^\perp \subseteq (\cal K \cap i\cal K)^\perp$.  If $\mathfrak{A}'$ is dense, then so is $(\cal K \cap i\cal K)^\perp$, in which case, $\cal K \cap i\cal K = \{0\}$, as desired. 
    
    We show (2) $\implies$ (3).  Assume, $\cal K \cap i\cal K = \{0\}$, then by \cite[Lemma 5.12]{RvD}, this implies $\cal K$ is the closure of the set of self-adjoint elements of $\mathfrak{A}$ and $\cal K' = i\cal K^{\perp}$.  Therefore $i\cal K^{\perp}$ is the closure of the set of self-adjoint elements of $\mathfrak{A}'$.  Thus, $\mathfrak{A}'$ is dense in $i\cal K^\perp + \cal K^\perp \subseteq (\cal K \cap i\cal K)^\perp$.  By assumption, the latter is all of $\cal H$.

    That (3) $\implies$ (4) is immediate from the definitions.

    Finally, we show (4) $\implies$ (1).  Assume the set of right bounded vectors is dense.  The set of right bounded vectors is closed under $\R_{\lambda}(\mathfrak{A})'$ by Proposition \ref{rightBoundedClosed}.  Assuming $v$ is right bounded, then $\pi'(v), \pi(v)^* \in \R_{\lambda}(\mathfrak{A})'$.  Therefore $w = \pi'(v_1)^*v_2$ is a right bounded vector for all $v_1, v_2 \in \cal H_{\rt}$.  Define $F_0 w = w^\flat = \pi'(v_2)^*v_1$.  Then we have
    \begin{align}
        \ip{a^\#}{w} & = \ip{a^\#}{\pi'(v_1)^*v_2}  = \ip{\pi'(v_1)a^\#}{v_2} \nonumber \\
        & = \ip{\pi(a)^*v_1}{v_2}  = \ip{v_1}{\pi(a)v_2} \nonumber \\
        & = \ip{v_1}{\pi'(v_2)a} = \ip{v_1}{\pi'(v_2)a}  \nonumber \\
        & = \ip{\pi'(v_2)^*v_1}{a} = \ip{w^\flat}{a} \nonumber
    \end{align}
    Therefore $F_0$ is a densely defined operator which is adjoint to $S_0 a = a^\#$.  It follows that $^\#$ is closable.
\end{proof}

\begin{remark}
    By \cite[Lemma 5.12]{RvD}, if $\mathfrak{A}$ is a left Hilbert algebra, then $\cal K$ defined as above is the closure of the real subspace of self-adjoint elements of $\mathfrak{A}$.
\end{remark}

\begin{defn}
    We call a pair $(\cal M, \Phi)$ a \textbf{weighted von Neumann algebra} if $\cal M$ is a von Neumann algebra and $\Phi$ is a faithful normal semifinite weight on $\cal M$.
\end{defn}

Let $(\cal M, \Phi)$ be a weighted von Neumann algebra.  Define
\[
P_{\Phi} := \{x \ : \ x \text{ is positive and } \Phi(x) < \infty\}.
\]
We also define the sets:
\[
    N_\Phi = \{x \in \cal M \ : \ x^*x\in P_\Phi\} \quad \text{ and } \quad D_\Phi = N_\Phi^*N_\Phi.
\]
One can show via a polarization argument (see \cite{Takesaki}) that
\[
    D_\Phi = \mathrm{span}\{a^*b \ : \ \Phi(a^*a) < \infty \text{ and }  \Phi(b^*b) < \infty \}.
\]
Notice that $\Phi$ can be extended to a linear functional on $D_\Phi$ and so we call $D_\Phi$ the \textbf{definition domain} of $\Phi$.  We get an obvious embedding
\[
    \eta_\Phi: D_\Phi \to \cal H_{\Phi}
\]
of $D_\Phi$ in its Hilbert space completion $\cal H_{\Phi}$.  Since $(yx)^*yx = x^*y^*yx \leq \|y\|x^*x$ and therefore $\Phi(x^*y^*yx) \leq \|y\|\Phi(x^*x)$, we have that $D_\Phi$ is a $^*$-closed left ideal of $\cal M$.  Therefore we can define a representation of $\cal M$ on $\cal H_{\Phi}$ as follows.  For any $a \in \cal M$, define the operator $\pi_{\Phi}(a) \in B(\cal H_{\Phi})$ by $\pi_{\Phi}(a)\eta_\Phi(b) = \eta_\Phi(ab)$ for all $b \in D_\Phi$ and extending to $\cal H_{\Phi}$ by continuity.

\begin{defn}
    Let $\cal M$ be a von Neumann algebra and let $\cal H_{\Phi}$ and $\pi_{\Phi}$ be defined as above.  Then $\pi_{\Phi}$ is called the \textbf{semicyclic representation} with respect to $\Phi$.
\end{defn}

We associate a full left Hilbert algebra to a weighted von Neumann algebra $(\cal M, \Phi)$ by taking $D_\Phi \cap D_\Phi^*$ together with the multiplication and involution given by
\[
    \eta_{\Phi}(a)\eta_{\Phi}(b) = \eta_{\Phi}(ab) \qquad \eta_{\Phi}(a)^\# = \eta_{\Phi}(a^*).
\]
One can check (see \cite{Takesaki}) that this defines a full Hilbert algebra.  Conversely, given a full left Hilbert algebra $\mathfrak{A}$, the left von Neumann algebra $\R_{\lambda}(\mathfrak{A})$ together with the faithful normal semifinite weight given for $x \in \R_{\lambda}(\mathfrak{A})$ positive by
\[
    \Phi_{\lambda}(x) = \begin{cases} 
      \sqrt{\ip{\xi}{\xi}} & \text{if } x^{1/2} = \xi \in \mathfrak{A} \\
      \infty & \text{otherwise }
    \end{cases}.
\]
We also define a faithful normal semifinite weight, called the \textbf{dual weight}, on the commutant by
\[
    \Phi_{\rho}(x) = \begin{cases} 
      \sqrt{\ip{\xi}{\xi}} & \text{if } x^{1/2} = \xi \in \mathfrak{A}' \\
      \infty & \text{otherwise }
    \end{cases}.
\]
It is seen by \cite[Chapter VII, Theorem 2.5 and Theorem 2.6]{Takesaki} that these two processes are essentially inverse to each other.

Our next theorem and corollary explain our use of the norms $\|\cdot\|_{\Phi}$ on $D_{\Phi}$ and $\|\cdot\|_{\Phi}^\#$ on $D_{\Phi} \cap D_{\Phi}^*$ where
\[
\|x\|_{\Phi} := \sqrt{\Phi(x^*x)}  \text{ and } \|x\|_{\Phi}^\# = \sqrt{\frac{\Phi(x^*x) + \Phi(xx^*)}{2}}.
\]
While it is elementary, the author could not locate a proof in the literature, so one is provided.

\begin{thm}\label{TheoremPhiMetrize}
    Let $(\cal M, \Phi)$ be a von Neumann algebra equipped with a faithful, normal, semifinite weight and let $\pi: \cal M \to \cal H$ be the associated semicyclic representation.  Then $\|x\|_\Phi = \sqrt{\Phi(x^*x)}$ metrizes the strong operator topology on $D_\Phi \cap \cal M_1$, the set of all $x \in D_\Phi$ such that $\|x\| \leq 1$.
\end{thm}

\begin{proof}
    $\cal M$ and $B(\cal H)$ are both von Neumann algebras and therefore, by Theorem , admit preduals.  Thus $\cal M_1$ and $B(\cal H)_1$ are WOT-compact by Banach-Alaoglu.  Since $\Phi$ is faithful and normal, $\pi$ is a continuous bijection from $\cal M_1$ to $\pi(\cal M_1) \subset B(\cal H)$.  A continuous bijection between compact Hausdorff spaces is a homeomorphism, so $\pi$ is a homeomorphism from $\cal M_1$ to $\pi(\cal M_1)$.  We know that $x_n \to x$ converges in the strong operator topology if and only $(x_n - x)^*(x_n - x) \to 0$ in the WOT.  Since $\pi$ is a $^*$-homomorphism, we conclude that a net $(x_n)$ in $\cal M_1$ strong-converges if and only $\pi(x_n)$ strong-converges.
    
    Now suppose that $(x_n)$ is a net in $D_\Phi \cap \cal M_1$ such that $\|x_n - x\| \to 0$ for some $x \in D_\Phi \cap \cal M_1$.  Let $\eta_\Phi(a)$ be a right bounded element of $\cal H$.  Then there is $a := \pi'(\eta_\Phi(a)) \in \pi(\cal M)'$ and we have:
    \begin{align}
        \|(x_n - x)\eta_\Phi(a)\|_\Phi^2 &= \ip{(x_n - x)\eta_\Phi(a)}{(x_n - x)\eta_\Phi(a)} \nonumber \\
        &= \ip{\pi'(a)\eta_\Phi(x_n - x)}{\pi'(a)\eta_\Phi(x_n - x)} \nonumber \\
        &= \ip{\pi'(a)^*\pi'(a)\eta_\Phi(x_n - x)}{\eta_\Phi(x_n - x)} \nonumber \\
        &= \|a\|^2\|x_n - x\|_\Phi^2 \to 0 \nonumber
    \end{align}
    Since right-bounded elements of $\cal H$ are dense, for every $\epsilon > 0$ and $v \in \cal H$, there is $\eta_\Phi(a)$ right bounded such that $\|v - \eta_\Phi(a)\|_\Phi < \epsilon$.  Whence
    \begin{align}
        |\|(x_n-x)\eta_\Phi(a)\|_\Phi - \|(x_n-x)v\|_\Phi| &\leq \|x_n - x\|_\Phi\|\eta_\Phi(a) - v\|_\Phi  \nonumber \\
        &\leq 2\epsilon.  \nonumber 
    \end{align}
    Thus $\|(x_n - x)v\| \to 0$ for all $v \in \cal H$.  Hence $\pi(x_n) \to \pi(x)$ in the strong operator topology and consequently $x_n \to x$ in the strong operator topology.  The converse is proven by the same considerations.
\end{proof}

\begin{cor}\label{SharpMetrize}
    $\|x\|_\Phi^\sharp$ metrizes the strong-$^*$ topology on norm bounded subsets of $D_\Phi \cap D_\Phi^*$.
\end{cor}

We close this section with a short review of the classical approach to Tomita-Takesaki theory for full left Hilbert algebras (equivalently, faithful normal semifinite weights).  Let $\mathfrak{A}$ be a full left Hilbert algebra with Hilbert space completion $\cal H$.  Recall that we write $S_0 a := a^\#$ for $a \in \mathfrak{A}$.  Since $\mathfrak{A}$ is a Hilbert algebra, $S_0$ is closable; we denote the closure of $S_0$ by $S$. We see that $S^{-1} = S$. Then there is a densely defined antilinear closed (unbounded) operator $F$ satisfying
\begin{itemize}
    \item $D(F)$ is the set of vectors $v \in \cal H$ such that the map $D(S) \to \cal H \ : \ w \mapsto \ip{v}{Sw}$ is bounded; and 
    \item $\ip{Sv}{w} = \ip{Fw}{v}$ for $v \in D(S)$ and $w \in D(F)$, so that $S^* = F$; and
    \item $F^{-1} = F$.
\end{itemize} 

Consider the linear positive non-singular self-adjoint operator $\Delta := FS$. Then $D(\Delta^{1/2}) = D(S)$. There is also an antilinear isometry $J$ of $H$ onto itself such that $\Delta$ and $J$ satisfy: $\ip{Jv}{Jw} = \ip{w}{v}$ for all $v,w \in \cal H$; $J^2 = 1$; and $J\Delta J = \Delta^{-1}$, $S = J\Delta^{1/2} = \Delta^{-1/2}J$, and $F = J\Delta^{-1/2} = \Delta^{1/2}J$.  By functional calculus, we can define a family of unitaries $\Delta^{it}$ for $t \in \bbR$. The fundamental theorem of Tomita-Takesaki theory is as follows.

\begin{thm}{\cite[Chapter VI, Theorem 1.19]{Takesaki}}\label{TomitaFund}
    Let $\mathfrak{A}$, $\Delta$, and $J$ be as above. 
    
    \begin{itemize}
        \item Let $t \in \bbR$ be given.
        \[
        \Delta^{it} \R_{\lambda}(\mathfrak{A})\Delta^{-it} = \R_{\lambda}(\mathfrak{A}) \quad J \R_{\lambda}(\mathfrak{A}) J = \R_{\lambda}(\mathfrak{A})' \quad 
        \]
        and similarly,
        \[
        \Delta^{it} \R_{\lambda}(\mathfrak{A})'\Delta^{-it} = \R_{\lambda}(\mathfrak{A})' \qquad J \R_{\lambda}(\mathfrak{A})' J = \R_{\lambda}(\mathfrak{A}).
        \]
        \item The one parameter unitary group $\{ \Delta^{it} \ : \ t \in \bbR\}$ acts on $\mathfrak{A}''$ and $\mathfrak{A}'$ as automorphisms and the modular conjugation $J$ maps $\mathfrak{A}''$ (resp. $\mathfrak{A}'$) onto $\mathfrak{A}'$ (resp. $\mathfrak{A}''$) anti-isomorphically in the sense that 
        \[
        J(vw) = (Jw)(Jv) \text{ for all } v,w \in \mathfrak{A}''
        \]
    \end{itemize}
\end{thm}



\section{Totally Bounded Elements}\label{SectionTotallyBounded}

In this section, we extend the notion of totally ($K$-)bounded elements, first explicitly defined in the \wstar-probability space regime in \cite{AGHS}, to weighted von Neumann algebras.  We prove several facts about these elements that are necessary for our axiomatization in Section \ref{SectionAxiomatization} to work.  Let $(\cal M, \Phi)$ be a weighted von Neumann algebra and consider the associated semicyclic representation $\pi$ on $\cal H$.

\begin{defn}
     We say $x \in \cal M$ is \textbf{totally $K$-bounded} if $\| x \| \leq K$ and both $\eta_{\Phi}(x)$ and $\eta_{\Phi}(x^*)$ are right $K$-bounded.  We say that $x \in M$ is \textbf{totally bounded} if there exists $K$ such that $x$ is totally $K$-bounded.  We  denote the set of totally bounded elements by $\cal M_{\tb}$ or $\cal H_{\tb}$ depending on the perspective we want to take.
\end{defn}

\begin{notation}
    For any $X \subseteq M$, we let $S_{K}(X)$ denote the set of totally K-bounded elements $x$ of $X$ such that, furthermore, $\|x\|_{\Phi} \leq K$.
\end{notation}

The proof of the following lemma is now essentially identical to the proof given in \cite[Lemma 2.7]{AGHS} so we omit it.

\begin{lem}\label{LemSortsComplete}
    For each $K$, the set $S_{K}(\cal M)$ is $\|\cdot\|_{\Phi}^\#$-complete.
\end{lem}

We also need the following to ensure that we still have enough totally bounded elements.  We do this with the following theorem.

\begin{thm} \label{tbDense}
    $\cal H_{\tb}$ is a dense subspace of $\cal H$.
\end{thm}

There are multiple ways to prove Theorem \ref{tbDense}.  One way would be to emulate the proof of \cite[Proposition 2.3]{AGHS}, which is the same result in the $\sigma$-finite setting.  We provide here a proof using Tomita algebras, which we now define.

\begin{defn}{\cite[Chapter VI, Definition 2.1]{Takesaki}}
     A left Hilbert algebra $\mathfrak{A}$ is called a \textbf{Tomita algebra} if there is a complex one parameter group $\{U(\alpha) \ : \ \alpha \in \bbC\}$ of automorphisms (which are here not necessarily $^*$-preserving) of $\mathfrak{A}$ with the following properties:
     \begin{itemize}
         \item The function $\alpha \mapsto \ip{U(\alpha)v}{w}$ for $\alpha \in \bbC$ is entire;
         \item $(U(\alpha)v)^\# = U(\overline{\alpha})v^\#$;
         \item $ \ip{U(\alpha)v}{w} = \ip{v}{U(-\alpha)w}$, for $\alpha \in \bbC$, and $v, w \in \mathfrak{A}$;
         \item $ \ip{v^\#}{w} = \ip{U(-i)w}{v}$ for $v, w \in \mathfrak{A}$.
     \end{itemize}
     The group $\{ U(\alpha) \ : \ \alpha \in \bbC \}$ is called the modular automorphism group of $\mathfrak{A}$.
\end{defn}

\begin{thm}{\cite[Chapter VI, Theorem 2.2]{Takesaki}}\label{TomitaDense}
    Given a full left Hilbert algebra $\mathfrak{A}$ with modular operator $\Delta$, define
    \[
    \mathfrak{A}_0 := \left\{ v \in \bigcap_{n \in \bbZ} D(\Delta^n) \ : \  \Delta^n v \in \mathfrak{A} \text{ for all } n \in \bbZ \right\}.
    \]
    then $\mathfrak{A}_0$ is a Tomita algebra with respect to $\Delta^{i\alpha}$ such that
    \begin{itemize}
        \item $\mathfrak{A}_0'' = \mathfrak{A}$, $\mathfrak{A}_0' = \mathfrak{A}'$ and $J\mathfrak{A}_0 = \mathfrak{A}_0$; and
        \item $\R_{\lambda}(\mathfrak{A}_0) = \R_{\lambda}(\mathfrak{A})$ and $\R_{\rho}(\mathfrak{A}_0) = \R_{\lambda}(\mathfrak{A})'$.
    \end{itemize}
\end{thm}

\begin{proof}[Proof of Theorem \ref{tbDense}]
    Consider $\mathfrak{A}_0$ defined as in Theorem \ref{TomitaDense}.  Notice that every element of $\mathfrak{A}_0$ is right bounded since $\mathfrak{A}_0$ is closed under $J$ which, by Theorem \ref{TomitaFund}, sends left bounded elements to right bounded elements and vice-versa, and $J^2 = 1$.  Since $\mathfrak{A}$ is also closed under $^\#$, it follows that $\mathfrak{A}_0$ consists of totally bounded elements.  Thus it suffices to show that $\mathfrak{A}_0$ is dense in $\cal H$.  To see this, let $v \in \cal H$ be given.  Define
    \[
    v_r := \sqrt{\frac{r}{\pi}}\int_{\bbR} e^{-rt^2}v \mathrm{dt}.
    \]
    Note that $v_r \in \mathfrak{A}_0$ for all $r > 0$.  This is because for all $z \in \bbC$, the vector
    \[
        v_r(z) = \sqrt{\frac{r}{\pi}}\int_{\bbR} e^{-r(t-z)^2}\Delta^{it}v\operatorname{dt}
    \]
    defines $\Delta^{iz}v_r := v_r(z)$ as an entire function in the variable $z$ by a change of variables.  Note also that $v_r \to v$ as $r \to 0$.  This proves the claim.
\end{proof}

We define $P$ to be the projection onto $\cal K$ and $Q$ to be the projection onto $i\cal K$.  By \cite[Lemma 5.12]{RvD}, $\cal K$ is the closure of the set of self-adjoint elements of $\eta_{\Phi}(\cal M)$.  We also define $R = P + Q$ and $TJ = P - Q$ as in \cite{RvD}.  Denote by $P'$, $Q'$ and $R'$ the analogous operators for $\cal M'$.

\begin{thm}
     If $x \in \cal M$ is left $K$-bounded, then $P\eta_{\Phi}(x)$ and $Q\eta_{\Phi}(x)$ are right $(2K^2+2)$-bounded.  Moreover, since $P$ is a projection, $\|P\eta_{\Phi}(x)\|_{\Phi} \leq \|\eta_{\Phi}(x)\|_{\Phi}$.
\end{thm}

\begin{proof}
    Examine the proofs of \cite[Lemma 5.4, Lemma 5.6 and Corollary 5.7]{RvD}.  There, it is seen that if $x \in \cal K$, then $\pi'(P\eta_{\Phi}(x)) = a^{-1}b$ where $a, b \in \cal \R_{\rho}(\mathfrak{A}')$ are as defined in \cite[Lemma 5.4]{RvD}.  It is clear from the definition that $\|b\| \leq 1$.  It is shown in the proof of \cite[Lemma 5.6]{RvD} that $\|a^{-1}\| \leq 1+\|\pi(x)\|^2$.  Thus, by decomposing general $x$ into real and imaginary parts, we see that if $x$ is left $K$-bounded, then $P\eta_{\Phi}(x)$ is right $(2K^2+2)$-bounded.  By the same reasoning, $Q\eta_{\Phi}(x)$ is right $(2K^2+2)$-bounded.
\end{proof}

\begin{thm} \label{PQtotalbounds}
    If $x \in \cal M$ is totally $K$-bounded, then $P\eta_{\Phi}(x)$ and $Q\eta_{\Phi}(x)$ are totally $(5K^2+4)$-bounded.  In particular, $P(S_K(\cal M)) \subseteq S_{(5K^2+4)}(\cal M)$.
\end{thm}

\begin{proof}
    If $x$ is totally $K$-bounded, then since $P' = (1-Q)$, it follows that $x$ is left $(5K^2+4)$-bounded.  Combining with the previous theorem gives the result.  The proof for $Q$ is nearly identical.
\end{proof}

\begin{cor}\label{Rtotalbound}
    If $x \in \cal M$ is totally $K$-bounded, then $R\eta_{\Phi}(x)$ is totally $(10K^2+8)$-bounded.  Moreover, $\|R\eta_{\Phi}(x)\|_{\Phi} \leq 2\|\eta_{\Phi}(x)\|_{\Phi}$.  In particular, $R(S_{K}(\cal M)) \subseteq S_{(10K^2 + 8)}(\cal M)$.
\end{cor}

Following \cite{Kad}, we define the following families of functions indexed by $a \in \bbR$:
\begin{align}
    h_{a}(t) &= (\cosh(t-a))^{-1} = \frac{2}{e^{t-a} + e^{a-t}}; \nonumber \\
    f_{a}(t) &= e^{-|t-a|}; \nonumber \\
    g_a(t) &= e^{-|t|} - \frac{e^{-|t-a|} + e^{-|t+a|}}{e^a + e^{-a}}.  \nonumber
\end{align}

The results in the rest of this section are very mild generalizations of their analogues in \cite{AGHS} and, as such, we omit the proofs.  We formulate them here for convenience and to make explicit the forms of the generalizations.

\begin{lem}\label{haTotalBound}
    For every $a \in \bbR$ and $x \in \cal M_{\tb}$, there exists $y \in \cal M_{\tb}$ such that we have $h_{a}(\log(\Delta))\eta_{\Phi}(x) = \eta_{\Phi}(y)$.  Furthermore, $\|y\| \leq \|x\|$ and $\|\pi'(\eta_{\Phi}(y))\| \leq \|\pi'(\eta_{\Phi}(x))\|$.
\end{lem}

As in \cite{AGHS}, we see that $h_a(\log(\Delta))\eta_{\Phi}(x) = \eta_{\Phi}(y)$ if and only if $2R(2-R)\eta_{\Phi}(x) = (e^{-a}(2-R)^2 + e^{a}R^2)\eta_{\Phi}(y)$.

\begin{lem}{\cite[Lemma 4.11]{Kad}}
    For all $a \in \bbR$ and $x \in \cal M$, there exists $y \in \cal M$ such that $f_{a}(\log(\Delta))\eta_{\Phi}(x) = \eta_{\Phi}(y)$.  Furthermore, $\|y\| \leq \|x\|$.
\end{lem}

\begin{defn}
    Let $\cal M_0$ be a $^*$-subalgebra $\cal M_0 \subseteq \cal M$ of a von Neumann algebra $\cal M$ such that:
    \begin{itemize}
        \item every element of $\cal M_0$ is totally bounded;
        \item the set of totally $1$-bounded elements of $\cal M_0$ is $\|\cdot\|_{\Phi}^\#$-complete; and
        \item $\cal M_0$ is closed under $h_{a}(\log(\Delta))$ for all $a \in \bbN$.
    \end{itemize}
    Then we  call $\cal M_0$ a \textbf{good subalgebra}.
\end{defn}

\begin{thm}\label{KadisonTypeTheorem}
    If $\cal M_0$ is a good subalgebra, then $\cal M_0 = \cal M_{\tb}$.
\end{thm}

\section{Axiomatization}\label{SectionAxiomatization}

In this paper, we only ever consider weights $\Phi$ such that $\Phi(1) \geq 1$ or $\Phi(1) = \infty$.  This is to avoid certain issues with ultraproducts becoming trivial.  Namely, one might consider $\cal M_i = \cal M$ and take a state $\vp$ on $\cal M$.  Then if we take $\Phi_i = \frac{1}{i} \vp$, we would get that the ultraproduct is $\{0\}$.  Our restriction is fine in practice because one could always re-scale the weight.  Very rarely in practice does one consider a bounded weight that is not a state, so this is not particularly objectionable.

We introduce the language $\cL_{\vNa}$ for weighted von Neumann algebras, whose symbols include:
\begin{enumerate}
    \item For each $n \in \bbN$, there is a sort $S_n$ with bound $2n$, whose intended interpretation is the set of $n$-bounded, right $n$-bounded elements $x$ of $D_{\Phi}$ with right $n$-bounded adjoint.  We let $d_n$ denote the metric symbol on $S_n$, whose intended interpretation is the metric induced by $\|\cdot\|_\Phi^\#$.
    \item For each $n\in \bbN$, binary function symbols $+_n$ and $-_n$ with domain $S_{n}^2$ and range $S_{2n}$ and whose modulus of uniform continuity is $\delta(\epsilon) = \epsilon$.  The intended interpretation of these symbols are addition and subtraction in the algebra restricted to the sort $S_n$.
    \item For each $n\in \bbN$, a binary function symbol $\times_n$ with domain $S_n^2$ and range $S_{n^2}$ and whose modulus of uniform continuity is $\delta(\epsilon)=\frac{\epsilon}{n}$.  The intended interpretation of these symbols is multiplication in the algebra restricted to the sort $S_n$.
    \item For each $n\in \bbN$, a unary function symbol $*_n$ whose modulus of uniform continuity is $\delta(\epsilon) = \epsilon$.  The intended interpretation of these symbols is for the adjoint restricted to each sort.
    \item For each $n\in\bbN$, the constant symbol $0_n$ which lies in the sort $S_n$.  The intended interpretation of these symbols is the element $0$.
    \item For each $n\in \bbN$ and $\lambda \in \bbC$, there is a unary function symbol $\lambda_n$ whose domain is $S_n$ and range is $S_{mn}$, where $m = \lceil|\lambda|\rceil$ and with modulus of uniform continuity $\delta(\epsilon) = \frac{\epsilon}{|\lambda|}$ when $\lambda \not= 0$.  The intended interpretation of these symbols is scalar multiplication by $\lambda$ restricted to $S_n$.
    \item For each $m,n\in \bbN$ with $m < n$, we have a unary function symbols $\iota_{m,n}$ with domain $S_m$ and range $S_n$ and whose modulus of uniform continuity is $\delta(\epsilon) = \epsilon$.  The intended interpretation of these symbols is the inclusion map between the sorts.
    \item For each $n\in \bbN$, we have a unary predicate symbol $\Phi_{n}$ whose range is $[0,n]$ and whose modulus of uniform continuity is $\delta(\epsilon) = \frac{\epsilon}{\sqrt{2}}$.  The intended interpretation of this symbol is the restriction of the weight to $S_n$.
    \item For each $n\in \bbN$, there is a predicate symbol $\vA_n$ on the sort $S_n$ taking values in $[0,4n]$ and with modulus of uniform continuity $\delta(\epsilon)=\frac{\epsilon}{\sqrt{2}}$.  The intended interpretation of this symbol is the distance associated to the norm $\|\cdot\|_\Phi$ from an element to the set of self-adjoint elements in $S_{(5n^2+4)}$.  The $\sqrt{2}$ in the denominator stems from the fact that $\|a\|_\Phi \leq \sqrt{2}\|a\|_\Phi^\#$ for all $a \in D_{\Phi} \cap D_{\Phi}^*$.
\end{enumerate}

\begin{remark}
    Note that the main difference between $\cal L_{\vNa}$ and $\cal L_{\text{\wstar}}$ from \cite{AGHS} is that the above lacks an identity element.  This is explained by the observation that \cite{AGHS} is really essentially restricting the above to unital full Hilbert algebras (where the vector corresponding to the unit is also assumed to have norm 1).
\end{remark}

\begin{enumerate}
    \item The usual algebraic axioms requiring the dissection of $M$ into sorts $S_n$ to be a $^*$-algebra.
    \item Axioms saying that $\Phi$ is a positive linear functional.
    \item Axioms saying the connecting maps preserve addition, multiplication, adjoints and $\Phi$.
    \item An axiom for each $n$ requiring that $d_n$ is defined by the norm
    \[
        \|x\|_\Phi^\# = \sqrt{\frac{\Phi(x^*x)+\Phi(xx^*)}{2}}.
    \]
    \item Axioms requiring that the elements of $S_n$ are $n$-bounded with $n$-bounded adjoint.  To be explicit, for each $k$ we have
    \[
        \sup_{x \in S_n}\sup_{y \in S_k} \Phi((yx)^*yx)\dminus n^2\Phi(y^*y) \qquad \sup_{x \in S_n}\sup_{y \in S_k} \Phi((xy)^*xy)\dminus n^2\Phi(y^*y).
    \]
    and similarly for $x^*$.  We also add an axiom saying that elements $x$ of $S_n$ have $\|x\|_{\Phi}^\sharp \leq n$:
    \[
        \sup_{x \in S_n} \Phi(x^*x + xx^*) \dminus 2n^2.
    \]
    \item Axioms expressing that the inclusion maps are isometries and have the correct ranges.  Specifically, for each $n$, and for all $k$ we add:
    \[
        \sup_{x \in S_1}  \bigl|\|\iota_n(x)\|-\|x\|\bigr|
    \]
    and
    \[
        \sup_{x \in S_1}\sup_{z\in S_k} \bigl|\Phi((zx)^*(zx)) -\Phi((z\iota_n (x))^*(z\iota_n (x))\bigr|.
    \]
    \item Axioms expressing that the inclusion $\iota_n: S_1\to S_n$ has range consisting of all $x\in S_n$ where x is left and right $1$-bounded with right $1$-bounded adjoint:
    \begin{equation*}
        \begin{split}
            \sup_{x \in S_n} \sup_{z \in S_k}\inf_{y \in S_1}&\max\bigl\{\|x - \iota_n(y)\| \dminus (\|x\| \dminus 1),\\
            &\|x - \iota_n(y)\| \dminus \bigl(\Phi((zx)^*(zx)) \dminus \,\Phi(z^*z)\bigr),\\
            &\|x - \iota_n(y)\| \dminus \bigl(\Phi((zx^*)^*(zx^*)) \dminus \,\Phi(z^*z)\bigr)\bigr\}.
        \end{split}
    \end{equation*}
    \item  Axioms expressing that $\vA_n$ represents the distance to the self-adjoint elements of $S_{(5n^2+4)}$.  More precisely, for each $n\in \bbN$, we include the axiom \[\sup_{x \in S_n} \left|\vA_n(x)-\inf_{y \in S_{(5n^2+4)}} \left\|x- \frac{y+y^*}{2}\right\|_\Phi\right|.\]
    \item For each $a \in \bbN$, we add the following $\forall\exists$ axioms saying that models are closed under $h_a(\log(\Delta))$.  Defining $m = 2\lceil e^{a} \rceil (10n^2+8)^2$, we add
    \[
        \sup_{x \in S_n}\inf_{y \in S_n} d_{S_{m}}(2R(2-R)x, (e^{-a}(2-R)^2 + e^{a}R^2)y).
    \]
    Here, $m$ reflects the natural sort for our terms to land in as compositions of functions.  To see that this does indeed imply closure under $h_a(\log(\Delta))$, see the discussion following Lemma \ref{haTotalBound}.
    \item We have asked that all of our weights majorize a state.  To this end, we add an axiom
    \[
        \sup_{x \in S_1} (\sup_{y \in S_1} \max\{d_{S_1}(xy, y), d_{S_1}(yx,y)\} \dminus (d_{S_1}(x,0) \dminus 1)).
    \]
    Note that this says that if $x$ is an element of $S_1$ that act like the identity on $S_1$ (and hence on all of $\cal H$ by linearity and density), then $\|x\|_{\Phi} = 1$.  Notice that if there is no such $x$, then $\Phi(1) > 1$ since $1$ has left and right norm $1$.
\end{enumerate}

\begin{remark}
    In axiom (9), we use the symbol $R$ to represent the operator defined in Section \ref{SectionTotallyBounded}.  This is not technically part of the language but we use it as a convenient abbreviation for an expression which defines it.  See Section \ref{SectionDefModular} for more details.
\end{remark}

Note that, assuming the above axiomatization is correct (see Theorem \ref{correctness} below), we could expand $T_{\vNa}$ to axiomatize:
\begin{itemize}
    \item \wstar-probability spaces by adding an axiom
    \[
    \inf_{x \in S_1}\sup_{y \in S_1} \max\{ |\Phi(x^*x) - 1|, d(xy, y), d(yx, y) \}
    \]
    saying there is an element $x$ with $\|x\|_{\Phi} = 1$ that acts as the identity;
    \item semifinite von Neumann algebras equipped with a tracial weight by adding an axiom
    \[
    \sup_{x,y \in S_1} |\Phi(xy) - \Phi(yx)|
    \]
    saying $\Phi$ is a tracial weight (we will call these ``tracial weighted von Neumann algebras''); or
    \item tracial von Neumann algebras by combining the previous two items.
\end{itemize}

In the first and last cases above, our formalism stays close to the existing axiomatizations of the respective classes.

\begin{remark}
    All of the axiomatizations we have considered above are computably enumerable.
\end{remark}

We now undertake the main theorem of this section.  In the interest of clarity, we break the proof into a series of claims.

\begin{thm}\label{correctness}
    The category of models of $T_{\vNa}$ is equivalent to the category of weighted von Neumann algebras with morphisms being weight-preserving embeddings with a conditional expectation.
\end{thm}

\begin{proof}[Proof of Theorem \ref{correctness}]

    As usual in continuous logic, we define the \textbf{dissection} associated to a weighted von Neumann algebra.  This is a functor from the category of weighted von Neumann algebras to $\Md(T_{\vNa})$.  Given a weighted von Neumann algebra $(\cal M, \Phi)$, form the associated semicyclic representation.  The sort $S_n$ of the dissection $\cal D(\cal M, \Phi)$ consists of the totally $n$-bounded elements of the semicyclic representation (or the associated full Hilbert algebra).  Interpreting all other symbols as intended, we see

    \begin{claim}
        For a weighted von Neumann algebra $(\cal M,\Phi)$, the dissection $\cal D(\cal M,\Phi)$ is a model of $T_{\vNa}$.  Moreover, in any dissection $\cal D(\cal M,\Phi)$, $\vA_n$ captures the distance from an element of $S_n(\cal M)$ to $\cal M_{\sa}$.
    \end{claim}

    Now we must consider the ``inverse'' functor, called the ``interpretation''.  We associate to any model of $T_{\vNa}$, a Hilbert algebra as follows.  We see later in this section that this Hilbert algebra is full and therefore can be considered as a weighted von Neumann algebra.

    Suppose that we have a model $A \in \Md(T_{\vNa})$ of the theory $T_{\vNa}$.  We begin by forming the direct limit $\mathfrak{A}_0(A)$ of the sorts $S_n(A)$ for $n \in \bbN$ via the embeddings $i_{k,n}$.  This is well-defined by axioms (6) and (7).  Using the interpretation of the function symbols on each sort, we see that $\mathfrak{A}_0(A)$ is naturally a complex $^*$-algebra by axiom (1).  Furthermore, using the predicate for the weight on each sort, one can define an inner product $\langle x,y \rangle := \Phi(y^*x)$ on $\mathfrak{A}_0(A)$.  We let $\cal H_0$ denote the Hilbert space completion of $\mathfrak{A}_0(A)$ with respect to this inner product.  For each $a \in \mathfrak{A}_0(A)$, the maps $b \mapsto ab: \mathfrak{A}_0(A) \to \mathfrak{A}_0(A)$ and $b \mapsto ba: \mathfrak{A}_0(A) \to \mathfrak{A}_0(A)$ extend to unique bounded linear operators $\pi(a): \cal H_0 \to \cal H_0$ and $\pi'(a): \cal H_0 \to \cal H_0$ respectively; satisfying $\|\pi(a)\|, \|\pi'(a)\| \leq n$ if $a \in S_n(A)$ by axiom (5).  Axiom (6) guarantees that each element belongs to the correct sort.  By construction, we have that right bounded elements are dense in $\cal H_0$.  Thus by Theorem \ref{HilbertEquivalentConditions}, $\mathfrak{A}_0(A)$ is a Hilbert algebra and consequently, the following definition makes sense.

    We define the \textbf{interpretation} $\mathfrak{A}_{A}$ of $A \in \Md(T_{\vNa})$ to be the full left Hilbert algebra $\mathfrak{A}_{A} := (\mathfrak{A}_0(A)){''}$ on $\cal H := \cal H_0$ generated by $A$ as in the discussion above.

    Note by definition that the direct limit $\mathfrak{A}_0(A)$ of the sorts of $A$ is strong$^*$-dense in the interpretation.  For dissections and interpretations to define an equivalence of categories, we need to show that for any $A \in \Md (T_{\vNa})$, the algebra $\mathfrak{A}_0(A)$ associated to $A$ is precisely the set of totally bounded elements of the interpretation of $A$.

    \begin{claim}
        If $x \in \mathfrak{A}_0(A)$, then $x$ is totally bounded in $\mathfrak{A}_{A}$.  In fact, if $x$ is totally $n$-bounded in $\mathfrak{A}_0(A)$, then $x$ is totally $n$-bounded in $\mathfrak{A}_{A}$.
    \end{claim}

    This follows immediately from axiom (5) and strong-density of $\mathfrak{A}_0(A)$ in $\mathfrak{A}_{A}$.  If $x \in S_n(A)$ such that $\|x\|_{\Phi} \leq 1$, $\|\pi(x)\|, \|\pi(x^*)\| \leq 1$ and $\|\pi'(x)\|, \|\pi'(x^*)\| \leq 1$ then since the inclusion maps have the correct images by axiom (6), the inclusions are isometries.  By axiom (7), and the fact that the sorts are complete, it follows that $x \in S_1(A)$.

    \begin{claim}
        $\mathfrak{A}_0(A)$ is closed under $P$, $Q$ and $R$.
    \end{claim}
    
    $\mathfrak{A}_0(A)$ is strong-dense in $\mathfrak{A}_{A}$ by assumption and $\|\cdot\|_\Phi$ metrizes the strong operator topology (see Theorem \ref{TheoremPhiMetrize}).  Thus the $\|\cdot\|_\Phi$-distance to the self-adjoint elements in $\mathfrak{A}_0(A)$ agrees with that in $\mathfrak{A}_{A}$.  Now, if $a \in S_n(\mathfrak{A}_0(A))$ then, by the previous claim, $a \in S_n(\mathfrak{A}_{A})$.  Thus $P(a)$ is totally bounded.  Moreover, axiom (8) tells us that $\vA_n(x)$ can be computed in $S_{(5n^2+4)}(\mathfrak{A}_0(A))$.  Thus, by definition of $\inf$, we have a sequence $(a_i)$ of self adjoint operators in $S_{(5n^2+4)}(\mathfrak{A}_0(A))$ such that $\|a_i - x\|_\Phi$ converges to $\vA_{n}(x)$.  Since $\|\cdot\|_\Phi$ metrizes the strong operator topology, we have that $(a_i)$ converges to $P(x)$.  By the uniform total bound of $5n^2+4$, it follows that $(a_i)$ strong-$^*$ converges to $P(x)$.  By Lemma \ref{LemSortsComplete} and Corollary \ref{SharpMetrize}, we have that $P(x) \in S_{(5n^2+4)}(\mathfrak{A}_0(A))$.  Thus $\mathfrak{A}_0(A)$ is closed under $P$ and, in turn, under $Q$.  Therefore, by taking linear combinations, $\mathfrak{A}_0(A)$ is closed under $R$, proving the claim.

    We also have, by axiom (9) that $\mathfrak{A}_0(A)$ is closed under $h_{a}(\log(\Delta))$ for every $a \in \bbN$.  Together with the previous discussion, we conclude that $\mathfrak{A}_0(A)$ is a good subalgebra of $\mathfrak{A}_{A}$. By Theorem \ref{KadisonTypeTheorem}, we conclude that:
    \begin{claim}
        $\mathfrak{A}_0(A)$ is precisely the set of totally bounded elements of $\mathfrak{A}_{A}$.
    \end{claim}

    At this point, we have shown that to a model $A$ of $T_{\vNa}$, we can associate a Hilbert algebra $\mathfrak{A}_0(A)$ of operators on a Hilbert space $\cal H$ and a faithful $^*$-representation $\pi: \mathfrak{A}_0(A) \to \cal B(\cal H)$ such that, setting $\cal M_A$ to be the strong closure of $\pi(\mathfrak{A}_0(A))$ and letting $\Phi_A$ denote the weight on $\cal M_A$ corresponding to the inner product on $\cal H$, we have that $\mathfrak{A}_0(A)$ is a good subalgebra for the full Hilbert algebra $\mathfrak{A}_{A}$ of $\cal M_A$.  Hence we have $\cal D(\cal M_A, \Phi_A)=A$.  In particular, if we start with a weighted von Neumann algebra $(\cal M, \Phi)$ and let $A := \cal D(\cal M, \Phi)$, we have (using that the totally bounded vectors in $\mathfrak{A}_{A}$ are dense in $\mathfrak{A}_{A}$ together with Theorem \ref{TheoremPhiMetrize}) that $(\cal M_A, \Phi_A)=(\cal M, \Phi)$ and $\cal D(\cal M_A,\Phi_A) = \cal D(\cal M,\Phi)$.  Finally, we observe that

    \begin{claim}
        Embeddings between weighted von Neumann algebras correspond to embeddings between the corresponding dissections.
    \end{claim}
    One direction of this claim is obvious; to see the other, suppose that $\cal D(\cal M,\Phi)\subseteq \cal D(\cal N,\psi)$.  By Takesaki's Expectation Theorem and the discussion in \cite{RvD}, we need to show that, for each $a \in \eta_{\Phi}(\cal M)$, the $\|\cdot\|_\Phi$-distance between $a$ and the self-adjoint elements of $\eta_{\Phi}(\cal M)$ is the same as the $\|\cdot\|_\psi$-distance between $a$ and the self-adjoint elements of $\eta_{\Psi}(\cal N)$.  However this follows from axiom (7) and axiom (8), the density of bounded elements, and the fact from Theorem \ref{PQtotalbounds} that for $a \in S_n(\cal M)$, the distance from $a$ to $\cal M_{\sa}$ is realized by an element of $S_{(5n^2+4)}(\cal M)$.
\end{proof}

\section{Definability of Modular Automorphisms}\label{SectionDefModular}

Our next main goal is to produce a definitional expansion of $T_{\vNa}$ which axiomatizes the modular group.  Recall that in the \wstar-probability space setting, we get the following immediately by Beth definability.

\begin{cor}{\cite[Corollary 7.2]{AGHS}} \label{AGHSMod}
    For all $t\in \bbR$, $\sigma_t$ is a $T_{\text{\wstar}}$-definable function.  Moreover, if $\vp_t$ is a $T_{\text{\wstar}}$-definable predicate defining $\sigma_t$, then the map $t\mapsto \vp_t$ is continuous with respect to the logic topology.
\end{cor}

However, in the weighted von Neumann algebra setting, we do not not yet know enough about the ultraproducts involved to conclude the analogous statement.  On the other hand, as in \cite{AGHS}, we can directly construct an expansion by definitions to include $\sigma_t$.

We begin by adding new functions symbols $\vP_n:S_n\to S_{(5n^2+4)}$ to our language and expand our dissections to interpret $\vP_n$ as the restriction to our sorts of the projection $P$ onto $\cal K$ as in Section \ref{SectionHilbertAlg} (also see \cite[Section 5]{RvD}.  To this end, we add the axioms:

\begin{enumerate}
    \item[(11)] $\sup_{x\in S_n} \vA_n(\vP_n(x))$; and
    \item[(12)] $\sup_{x\in S_n}\sup_{y\in S_1} \vp(y^*y(x-\vP_n(x)))$.
\end{enumerate}

We also add symbols $\vQ_n:S_n\to S_{(5n^2+4)}$ and $\vR_n:S_n\to S_{(10n^2+8)}$ to be interpreted as the corresponding restrictions of the projection onto $i\cal K$ and $R := P + Q$.  We add to our theory:

\begin{enumerate}
    \item[(13)] $\sup_{x\in S_n}d_{(5n^2+4)}(\vQ_n(x),i\vP_n(-ix))$; and
    \item[(14)] $\sup_{x\in S_n}d_{(10n^2+8)}(\vR_n(x),\vP_n(x)+\vQ_n(x))$.
\end{enumerate}

It now follows that models of axioms (1)-(14) are dissections of weighted von Neumann algebras with $\vP$ interpreted in a dissection as the real projection onto the closure of $\cal M_{\sa}$ and $\vQ$ as the corresponding projection onto the closure of $i\cal M_{\sa}$. 

By Theorem \ref{TomitaFund}, we have that $\Delta^{it}$ acts isometrically on both $\cal M$ and $\cal M'$. Thus:

\begin{lem}
    If $a\in S_n(\cal M)$, then for all $t\in \bbR$, we have $\Delta^{it} \eta_{\Phi}(a) = \eta_{\Phi}(b)$ for a unique $b \in S_n(\cal M)$.
\end{lem}

Let $X := \spec (R) \subseteq [0,2]$ denote the spectrum of $R$.  Define for $t \in \bbR$, the function $f_t : X \to \bbC$ defined by $f_t(x) = x^{it}$.  There exists polynomial functions $f_{t,n}$ on $X$ with coefficients from $\bbQ(i)$ such that $\|f_t - f_{t,n}\|_\infty < \frac{1}{n}$.  Moreover, these polynomials can be found effectively in the sense that the map taking $(t,n)$ and returning the coefficients of $f_{t,n}$ from $\bbN^2$ to $\bbQ(i)^{<\omega}$ is a computable map.  The next lemma is straightforward from the definition of $\Delta^{it}$:

\begin{lem}
    For all $m,n\geq 1$, we have
    $$\|\Delta^{it} - f_{t,m}(2-R)f_{-t,n}(R)\| < \frac{1}{m}\|f_{-t}\|_\infty + \frac{1}{n}\|f_{t,m}\|_\infty.$$
\end{lem}

Denote the quantity on the right hand side of the inequality above by $\delta_{t,m,n}$. Note that the map $(t,m,n)\mapsto \delta_{t,m,n}$ is computable.  Set $q_{t,n}\in \bbN$ to an integer such that $(f_{-t,n}(R_1))(S_1) \subseteq S_{q_{t,n}}$.  For each $t \in \bbQ$, we add a symbol $\mathbf{\Delta}^{it} : S_1 \to S_1$ to the language and continue enumerating $T_{\vNa-\md}$ by adding the following axioms:

\begin{itemize}
    \item[(15)] $\sup_{x\in S_1}[\|\mathbf{\Delta}^{it}(x)-f_{t,m}(2-\vR_{q_{t,n}})(f_{-t,n}(\vR_1)(x)))\|_\Phi \dminus \delta_{t,m,n}].$
\end{itemize}

To the axioms in (15) completely rigorous, we need to plug $\Delta^{it}(x)$ into appropriate inclusion mappings.  We can do this because $\vR$'s takes values in sorts $S_p$ with $p$ predictably depending on $t$, $m$, and $n$.

The final lemma of this section is immediate from Theorem \ref{TomitaFund}.

\begin{lem}
    For each $a \in S_n(\cal M)$ and $t \in \bbR$, we have $\sigma^{\Phi}_t(a)\in S_n$.
\end{lem}

We complete our axiomatization $T_{\vNa-\md}$ by adding function symbols $\mathbf{\sigma}_{t}$ to the language and the following axioms:
\begin{enumerate}
    \item[(16)] $\sup_{a,x\in S_1} d_1(\mathbf{\sigma}_{t}(a)x,\mathbf{\mathbf{\Delta}}^{it}(a\cdot \mathbf{\Delta}^{-it}(x)))$.
\end{enumerate}

We have now described a language $\cal L_{\vNa-\md}$ extending the language $\cal L$ and an $\cal L_{\vNa-\md}$-theory $T_{\vNa-\md}$ extending $T_{\vNa}$ for which the following theorem holds:

\begin{thm}\label{moddefinable}
    The category of models of $T_{\vNa-\md}$ consists of the dissections of weighted von Neumann algebras with the symbols $\vP$, $\vQ$, $\vR$, and $\mathbf{\Delta^{it}}$ interpreted as above (restricted to their appropriate sorts) and with the symbols $\sigma_t$ interpreted as the modular automorphism group of the state (restricted to the sort $S_1$).  Moreover, the theory $T_{\vNa-\md}$ is effectively axiomatized.
\end{thm}

\section{Hilbert Algebra Ultraproducts}\label{SectionHAUltraproduct}

We give a natural generalization of Ocneanu ultraproducts to Hilbert algebras.  The primary aim of this section is to show that this ultraproduct of Hilbert algebras agrees with the metric structure ultraproduct for $T_{\vNa}$.

\begin{defn}
    Let $(\mathfrak{A}_i, \cal H_i)$ be a family of full left Hilbert algebras and let $\Phi_i$ be the induced faithful normal semifinite weight on the left von Neumann algebra $\cal R(\mathfrak{A}_i)$.  Define $\ell_{\Phi}^\infty(\mathfrak{A}_i)$ to be the set of all sequences $(x_i)$ such that $\|x_i\|_{\Phi_i}$ and $\|\pi(x_i)\|$ are both uniformly bounded.  Define the subalgebra
    \[
        \cal I_{\cU} = \bigl\{(x_i) \in \ell_{\Phi}^\infty(\mathfrak{A}_i) \ : \ \lim_{i \to \cU}\|x_i\|_{\Phi_i}^{\sharp} = 0\bigr\}
    \]
    and its two-sided normalizer
    \[
        \cal N_{\cU} = \{(x_i) \in \ell_{\Phi}^\infty(\mathfrak{A}_i) \ : \ (x_i) \cal I_{\cU} \subseteq \cal I_{\cU} \text{ and } \cal I_{\cU} (x_i)  \subseteq \cal I_{\cU}\}.
    \]
    Then the \textbf{Hilbert algebra ultraproduct} $\prod_{\HA}^{\cU} (\mathfrak{A}_i, \cal H_i)$ is defined to be $\cal N_{\cU}/\cal I_{\cU}$ together with its Hilbert space completion.
\end{defn}

The next proposition is clear.

\begin{prop}\label{HAandOcStates}
    The Hilbert algebra ultraproduct agrees with the Ocneanu ultraproduct when every $\Phi_i$ is a state.
\end{prop}

We would like to analyze the more general normal faithful semifinite weight case.  First, recall the F\'ejer kernel $F_{a} \in L^1(\bbR)$ defined for $a > 0$ by
\[
F_a(t)= \frac{a}{2\pi}1_{\{t=0\}}+\frac{1-\cos(at)}{\pi a t^2}1_{\{t\not=0\}}.
\]
and its Fourier transform 
\[
\widehat{F_a}(\lambda)= \left(1 - \frac{|\lambda|}{a}\right) 1_{\{|\lambda| \leq a\}} + (0) 1_{\{|\lambda| > a\}}.
\]
At this point, we recall some Arveson spectral theory, particularly the spectral subspaces $M(\sigma^\Phi,[-a,a])$, recording some of their properties that we need. See \cite{AH} and \cite[Chapter XI]{Takesaki} for a more thorough review.

\begin{prop}\label{modularsubspaces} Let $x \in M(\sigma^\Phi,[-a,a])$ for some $a > 0$ be given.
\begin{enumerate}
    \item The map $t\mapsto \sigma^\Phi_t(x)$ extends to an entire function $\mathbb{C}\to \cal M$; moreover, for any $z\in \mathbb{C}$, there is a constant $C_{a,z}$ depending only on $a$ and $z$ such that $\|\sigma^\Phi_z(x)\|\leq C_{a,z}\|x\|$.
    \item We have $x\in \cal M_{\tb}$ and $\|x\|_{\rt},\|x^*\|_{\rt}\leq C_{a,i/2}\|x\|$.
\end{enumerate}
\end{prop}

\begin{proof}
The proof of item (1) first appears for faithful normal semifinite weights in \cite[Lemma 4.2]{Ha79} and is reiterated in the state case in \cite[Lemma 4.13]{AH}.  The proof consists of estimates using the F\'ejer kernel.

Toward item (2), suppose $x \in M(\sigma^\Phi,[-a,a])$ and take $y \in \eta_{\Phi}(\cal M)$; we show that $\|y\eta_{\Phi}(x)\|_{\Phi} \leq C_{a,i/2} \cdot \|x\| \cdot \|\eta_{\Phi}(y)\|_{\Phi}$.  To see this, we calculate as follows:
\begin{align}
\|y\eta_{\Phi}(x)\|_{\Phi} &= \|(JyJ)(S\Delta^{-1/2}\eta_{\Phi}(x))\|_{\Phi} \nonumber && \text{since $J$ is isom. and $JS\Delta^{-1/2} = 1$} \\ \nonumber
            &=\|(JyJ)S\sigma^\Phi_{i/2}(\eta_{\Phi}(x))\|_{\Phi} && \text{by def. of $\sigma^\Phi_{i/2}$} \\ \nonumber
            &= \|(\sigma^\Phi_{i/2}(x))^*(\eta_{\Phi}(Jy))\|_{\Phi} && \text{by Theorem \ref{TomitaFund}}\\ \nonumber
            &\leq \|\sigma_{i/2}(x)^*\|\cdot \|\eta_{\Phi}(Jy)\|_{\Phi} && \text{by def. of operator norm}\\ \nonumber
            &\leq C_{a,i/2}\cdot \|x\| \cdot \|\eta_{\Phi}(y)\|_{\Phi}.  \nonumber
\end{align}
The result for $x^*$ since $M(\sigma^\Phi,[-a,a])$ is closed under adjoints.
\end{proof}

\begin{lem}\label{modcommutelem2}
    If there exists $a > 0$ and $K > 0$ such that $x_n \in M(\sigma^{\Phi_{n}},[-a,a])$ and $\|x_n\| \leq K$, then $x_n \in \cal N_{\cU}$.
\end{lem}

\begin{proof}
     Assume there exists $a > 0$ and $K > 0$ such that $x_n \in M(\sigma^{\Phi_{n}},[-a,a])$ and $\|x_n\| \leq K$.  Let $y_n \in \cal I_{\cU}$ be given.  The same computation as in the previous proof gives
     \[
     \|(xy)^*\|_{\Phi_n} \leq C_{a, -i/2}\|x_n\|\|y_n\|_{\Phi_n}.
     \]
     But $C_{a, -i/2}\|x_n\|\|y_n\|_{\Phi_n}$ goes to $0$ along $\cU$.  Similarly $\|x_n y_n\|_{\Phi_n} \to 0$ as $n \to \cU$.
\end{proof}

Now we state the main result of this section.

\begin{thm}\label{ultrapres}
Let $(\cal M_i,\vp_i)$ be a family of weighted von Neumann algebras for all $i \in I$ and $\cU$ is be ultrafilter on $I$.  Then $\mathcal{D}\left(\prod_{\HA}^{\cU} (\cal M_i,\Phi_i)\right)\cong \prod_\cU \mathcal{D}\left(\cal M_i,\Phi_i\right)$.
\end{thm}

We need the following generalization of \cite[Proposition 4.11]{AH}.

\begin{thm}\label{ultracontinuous}
    Let $(x_i) \in \ell_{\Phi}^\infty(\mathfrak{A}_i)$.  Then the following statements are equivalent:
    \begin{enumerate}
        \item $(x_i) \in \cal N_{\cU} := \{(x_i) \in \ell_{\Phi}^\infty(\mathfrak{A}_i) \ : \ (x_i) \cal I_{\cU} \subseteq \cal I_{\cU} \text{ and } \cal I_{\cU} (x_i)  \subseteq \cal I_{\cU}\}$.
        \item For any $\epsilon > 0$, there exists $a > 0$ and $(y_i) \in \ell_{\Phi}^\infty(\mathfrak{A}_i)$ such that
        \begin{itemize}
            \item $\lim_{i \to \cU} \|x_i - y_i\|^\#_{\Phi_i} < \epsilon$; and
            \item $y_i \in M(\sigma^{\Phi_i}, [-a, a])$ for all $i \in I$.
        \end{itemize}
    \end{enumerate}
\end{thm}

\begin{proof}
    $(1) \implies (2)$: Let $(x_i) \in \cal N_\cU$ and put $x := (x_i)_\cU$.  Also, define
    \[
        x_a := \sigma^{\Phi^{\cU}}_{F_a} (x) \in \prod_{\HA}^{\cU} (\cal M_i,\Phi_i).
    \]    
    Define $f : t \mapsto \|x-\sigma^{\Phi^{\cU}}_t(x) \|^\#_{\Phi^\cU}$.  Since $f$ is continuous and bounded, we have
    \begin{align}
        \| \eta_{\Phi}(x_a - x) \|_{\Phi^{\cU}} &= \| \int_{\bbR} F_a(t) (\eta_{\Phi}(\sigma^{\Phi^{\cU}}_t (x) - x) \operatorname{dt}) \|_{\Phi^\cU}^\# \nonumber \\
        &= \int_{\bbR} F_a(t) \|\eta_{\Phi}(x - \sigma^{\Phi^{\cU}}_t(x)) \|_{\Phi^\cU}^\# \operatorname{dt}.  \nonumber
    \end{align}
    This expression goes to $0$ as $a \to \infty$ and therefore we have
    \[
    \lim_{a\to \infty} \| x_a-x \|^\#_{\Phi^{\cU}} = 0.
    \]
    Therefore there exists $a > 0$ such that $y := \sigma^{\Phi^{\cU}}_{F_a} (x)$ satisfies $\|\eta_{\Phi}(y - x)\|_{\Phi^{\cU}} < \epsilon$.  We have by Lemma \ref{modcommutelem2}, $y = (y_i)_\cU$, where $y_i = \sigma^{\Phi_i}_{F_a} (x_i)$ for $i \in I$ and
    \[
    \| y \| \|F_a\|_1 \|x\| = \|x\|.
    \]
    Therefore $(y_i)$ is a sequence satisfying $(2)$.  Note also that $\|y_i\| \leq \|x_i\|$ for $i \in I$.

    $(2) \implies (1)$: Take $x = (x_i) \in \ell_{\Phi}^\infty (\cal M_i)$ as in $(2)$.  Let $\epsilon > 0$.  Then by Lemma \ref{modcommutelem2} and by assumption, there is $y = (y_i) \in \cal N_\cU$ such that $\lim_{i \to \cU} \| x_i - y_i \|^\#_{\Phi_i} < \epsilon$.  Taking $m = (m_i) \in \cal I_\cU$ with $\sup_{n \geq 1} \|m_i\| \leq 1$, we have
    \[
    \lim_{i \to\cU} \|(x_i m_i)^*\|_{\Phi_i} \leq \lim_{i \to\cU} ( \|m^*_i\| \|x^*_i - y^*_i\|_{\Phi_i} + \|m^*_ix_i^*\|_{\Phi_i}) \leq \epsilon.
    \]
    Since $\epsilon > 0$ is arbitrary, we have $xm \in \cal I_{\cU}$ so $x \cal I_{\cU} \subseteq \cal I_{\cU}$.  Similarly, we have $\lim_{i \to \cU} \| m_i x_i \|_{\Phi_i} = 0$ so $mx \in \cal I_{\cU}$.  We conclude that $x \in \cal N_\cU$ as was required.
\end{proof}

\begin{proof} [Proof of Theorem \ref{ultrapres}]

First suppose that $(a_i)^\bullet \in \prod_\cU S_n(M_i)$.  It is clear that $(a_i) \in \ell_{\Phi}^\infty (\cal M_i,I)$; we wish to show that $(a_i)\in \cal M$.  To see this, suppose that $(m_i) \in \cal I_{\cU}$; we show that $(a_i m_i), (m_i a_i) \in \cal I_{\cU}$.  Suppose, towards a contradiction, that $(a_i m_i) \notin \cal I_{\cU}$ and set $L := \lim_\cU \|a_i m_i\|_{\Phi_i}^\# \neq 0$.  Consider the set of $i \in I$ for which $\|m_i\|_{\Phi_i}^\# < L/2n$, which belongs to $\cU$ since $(m_i) \in \cal I_{\cU}$.  For a $\cU$-large subset of these $i$, we have that $\|a_i m_i\|_{\Phi_i}^\# > L/2$; for these $i$, we have $\|a_i\| > n$ or $\|a_i^*\|_{\rt} > n$, which is a contradiction.  To see this, assuming $\|a_i\|, \|a_i^*\|_{\rt} \leq n$,
\begin{align}
    \|a_i m_i\|_{\Phi_i}^\# &= \sqrt{\frac{\|a_i m_i\|_{\Phi_i}^2 + \|m_i^* a_i^*\|_{\Phi_i}^2}{2}} \nonumber \\
    &\leq \sqrt{\frac{\|a_i\|^2 \|m_i\|_{\Phi_i}^2 + \|a_i^*\|_{\rt}^2 \|m_i^*\|_{\Phi_i}^2}{2}} \nonumber \\
    &\leq n\|m_i\|_{\Phi_i}^\#, \nonumber 
\end{align}
as claimed.  The proof that $(m_i a_i) \in \cal I_{\cU}$ proceeds similarly.  It is now clear that $(a_i)^\bullet = (a_i)^\star$ and that this element is in $S_{n}(\prod_{\HA}^{\cU}(\cal M_i,\Phi_i))$.

Now we have that the von Neumann algebra associated to $\prod_\cU \cal D(\cal M_i,\Phi_i)$ embeds into $\prod_{\HA}^{\cU} (\cal M_i, \Phi_i)$; it suffices to show that this embedding is actually onto.  It suffices to show that, given any $x \in \prod_{\HA}^{\cU} (\cal M_i,\Phi_i)$ with $\|x\| = 1$ and any $\epsilon > 0$, there is $n > 0$ and $y \in \prod_\cU S_n(M_i)$ such that $\|x-y\|^\#_{\Phi^\cU}< \epsilon$.  Write $x = (x_i)^\star$ and take $y = (y_i)^\star$ as in Theorem \ref{ultracontinuous} for some $a > 0$.  We show that $y \in \prod_\cU S_n(M_i)$ for some $n > 0$.  Assuming as above that $\|y_i\| \leq 1$ for all $i \in I$, by item (2) of Proposition \ref{modularsubspaces}, we have $\|y_i\|_{\rt}, \|y_i^*\|_{\rt} \leq C_{a,i/2}$ for all $i \in I$ and thus $y \in \prod_\cU S_n(M_i)$ when $n \geq C_{a,i/2}$.
\end{proof}

The next two theorems are a reward for our model-theoretic efforts.

Proposition \ref{ultrapres} and Theorem \ref{moddefinable} imply $(\sigma_t^{\Phi_i}(x_i))^\bullet = \sigma_t^{\Phi^{\cU}}(x_i)^\bullet$ whenever $x_i$ is totally $K$-bounded.  Strong density of totally bounded elements and strong continuity of modular automorphisms yields:

\begin{thm}\label{ultramod}
    The modular automorphism group commutes with the Hilbert algebra ultraproduct.
\end{thm}

Combining Theorem \ref{ultrapres} and Theorem \ref{correctness} gives the final theorem of this section.

\begin{thm}\label{HAWeight}
    The weight defined by
    \[
    \Phi_{\lambda}(x) = \begin{cases} 
      \sqrt{\ip{\xi}{\xi}} & \text{if } x^{1/2} = \xi \in \mathfrak{A} \\
      \infty & \text{otherwise}
    \end{cases}
    \]
    on the right von Neumann algebra $\R_{\lambda}(\prod_{\HA}^{\cU} (\mathfrak{A}_i, \cal H_i))$ of the Hilbert algebra ultraproduct is a faithful normal semifinite weight.
\end{thm}

\section{Generalized Ocneanu Ultraproducts}\label{SectionOcneanu}

Inspired by the work on continuous points of ultraproducts in \cite{BYG}, the author investigated the continuous points of Groh-Raynaud ultraproducts.  It is seen in \cite[Theorem 1.5]{MT} that the continuous part gives exactly the Ocneanu ultraproduct in the case that all weights involved are states.  In this section, we see that this result generalizes to weighted von Neumann algebras, giving another new characterization of the ultraproduct in $\Md(T_{\vNa})$.

Before going further, we point out that the Hilbert algebra ultraproduct does not in general agree with the Groh-Raynaud ultraproduct.  By \cite[Example 4.7]{AH}, we see that the Groh-Raynaud ultraproduct does not commute with modular automorphisms, and in fact, the ultralimit of the modular automorphisms is discontinuous on the Groh-Raynaud ultraproduct.  On the other hand, the Hilbert algebra ultraproduct commutes with modular automorphisms by Theorem \ref{ultramod}.  In this section, it is made clear that this is essentially the only difference, namely in Corollary \ref{GOisCornerOfGR}.  We now begin to make this precise.

\begin{defn}
    Let $(\cal M_i, \Phi_i)$ be a family of weighted von Neumann algebras and let $\cal H_i$ be the associated family of semicyclic representations.  Let $(v_i)$ be a sequence of vectors such that $v_i \in \eta_{\Phi_i}(\cal M_i)$ for all $i$.  We call $(v_i)$  a \textbf{$\sigma$-continuous point} if $(\Delta^{it}_{\Phi_i}v_i)^\bullet$ is a continuous function of $t$ with respect to $\|\cdot\|_{\Phi}^\#$ on the Hilbert space ultraproduct $\prod^{\cU} \cal H_{\Phi_i}$.
\end{defn}

\begin{defn}
    Let $(\cal M_i, \Phi_i)$ be a family of weighted von Neumann algebras and let $\cal H_i$ be the associated family of semicyclic representations.  Let $(x_i)$ be a sequence of elements such that $x_i \in \cal M_i$ for all $i$.  We call $(x_i)$  a \textbf{$\sigma$-continuous point} if $(\sigma_{t}^{\Phi_i}(x_i))^\bullet$ is a continuous function of $t$ with respect to the strong$^*$ topology in the Groh-Raynaud ultraproduct $\prod_{\GR}^{\cU} (\cal M_i, \cal H_{\Phi_i})$.
\end{defn}

We aim to capture the set of $\sigma$-continuous points of the Groh-Raynaud ultraproduct.  We should define a faithful normal semifinite weight on this algebra.  Recall that given a family $(\cal M_i, \varphi_i)$ of von Neumann algebras $\cal M_i$ together with uniformly bounded linear functionals $\varphi_i$ on them, we can define a positive linear functional $\varphi_{\cU}$ on the Groh-Raynaud (and, a fortiori, Ocneanu) ultraproduct as
\[
    \varphi_{\cU}(x_i) = \lim_{i \to \cU} \varphi_i(x_i)
\]
and in fact, this is how the ultraproduct of states is defined.  The following generalization of \cite[Corollary 3.25]{AH} tells us that any positive linear functional on the ultraproduct can be achieved this way.

\begin{cor}\label{functionalsOnGRCor}
    Let $\varphi$ be a positive linear functional on the Groh-Raynaud ultraproduct $\prod_{\GR}^{\cU} (\cal M_i, \cal H_i)$.  Then there exists a uniformly bounded sequence $\varphi_i$ of positive linear functionals such that $\varphi = (\varphi_i)_{\cU}$.
\end{cor}

Note that this definition does not make sense in the case of a weight.  Indeed, consider $I = \bbN$ and take $\cal M_i = \cal M$ a fixed von Neumann algebra and let $\Phi$ be a faithful normal state on $\cal M$.  Take the faithful normal semifinite weight $\Phi_i = i^2\Phi$ for all $i$.  Now considering the sequence $x_i = \frac{1}{i}1$, we have that $(x_i)^\bullet = (0)^\bullet$ but
\[
    \lim_{i \to \cU} \Phi_i(x_i) = \lim_{i \to \cU} i = \infty
\]
and therefore this construction is ill-defined.

\begin{defn}
    Let $(\cal M_i, \Phi_i)$ be a family of weighted von Neumann algebras.  The \textbf{ultraproduct weight} $\Phi_{\cU}$ of $\Phi_i$ is the weight on the Groh-Raynaud ultraproduct $\prod_{\GR}^{\cU} (\cal M_i, \cal H_{\Phi_i})$ defined by
    \[
        \Phi_{\cU} = \sup \{\vp = (\vp_i)_{\cU} \ : \ \vp_i \in (\cal M_i)_{*}^+ \text{ uniformly bounded such that } \vp_i \leq \Phi_i\}.
    \]
\end{defn}

\begin{remark}
    Our definition of ultraproduct weight does not agree with the one given in \cite[Definition 4.25]{AH}.  The one given there only makes sense when we begin with an Ocneanu ultraproduct of \wstar-probability spaces.  The definition we give is much closer to, but still different from, the one in \cite{Cas}.  We see later in this section via model-theoretic techniques that our ultraproduct weight is always faithful normal semifinite on the generalized Ocneanu ultraproduct.
\end{remark}

Forget for the duration of this paragraph that $\Phi_i$ is assumed to be faithful normal semifinite.  If each $\Phi_i$ is assumed to be normal, it is clear by Corollary \ref{functionalsOnGRCor} and the main result of \cite{HaNormal}, that $\Phi_{\cU}$ is normal.  This observation inspired the definition above.  Caspers considers the same construction but for arbitrary normal weights in the unpublished pre-print \cite{Cas}.  There, it seems that much of the difficulty lies in finding a corner of the Groh-Raynaud ultraproduct on which $\Phi_{\cU}$ is faithful and semifinite.  Much of this difficulty is bypassed here because model theory gives us Theorem \ref{HAWeight}.

The following can be seen directly from the definitions.

\begin{prop}\label{ultraweights}
    If $\Phi_i$ is a sequence of faithful normal states, then
    \[
    \Phi_{\cU}(x_i) = \lim_{i \to \cU} \Phi_i(x_i)
    \]
    so that the ultraproduct weight agrees with the usual ultraproduct state when every $\Phi_i$ is a state.
\end{prop}

\begin{defn}\label{DefOc}
    Let $(\cal M_i, \Phi_i)$ be a family of weighted von Neumann algebras.  We define the \textbf{generalized Ocneanu ultraproduct} $\prod^{\cU}_{\Oc}(\cal M_i, \Phi_i)$ to be the subset of the Groh-Raynaud ultraproduct consisting of elements $(x_i)^\bullet$ for which $(\sigma_{t}^{\Phi_i}(x_i))^\bullet$ is strong$^*$-continuous in the variable $t$.  It is equipped with the restriction of the ultraproduct weight $\Phi_{\cU}$ to it.
\end{defn}

We now show how to view the Hilbert algebra ultraproduct as a corner of the Groh-Raynaud ultraproduct.  Fix a family $(\cal M_i, \Phi_i)$ of weighted von Neumann algebras.  Denote by $\prod_{\HA}^{\cU}(\cal M_i, \Phi_i)$ the Hilbert algebra ultraproduct of the associated Hilbert algebras and denote by $\Phi^\cU$ the weight thereupon given in Theorem \ref{HAWeight}.  Let $\cal H_{\Phi_i}$ be the Hilbert space associated to the semicyclic representation of $(\cal M_i, \Phi_i)$ and note that $\cal M_i$ acts standardly on $\cal H_{\Phi_i}$.  By \cite[Theorem 3.18]{AH}, we have that $\prod^{\cU}_{GR} (\cal M_i, \cal H_{\Phi_i})$ acts standardly on the Hilbert space ultraproduct $\prod^{\cU} \cal H_{\Phi_i}$.

Inspired by \cite[Theorem 3.7]{AH}, define $W: L^2\left(\prod_{\HA}^{\cU}(\cal M_i, \Phi_i),\Phi^\cU\right) \to \prod^{\cU} (\cal H_{\Phi_i})$ by extending
\[
W(\eta_{\Phi^\cU}((x_i)^\bullet)) := (\eta_{\Phi_i}(x_i))^\bullet \qquad \text{for } (x_i)^\bullet \in \eta_{\Phi^\cU}\left(\prod_{\HA}^{\cU}(\cal M_i, \Phi_i)\right)
\]
to $L^2\left(\prod_{\HA}^{\cU}(\cal M_i, \Phi_i),\Phi^\cU\right)$ by density.  Clearly, $W$ is an isometry.  Define also
\[
p_{\Phi}: \prod^{\cU} (\cal H_{\Phi_i}) \to L^2\left(\prod_{\HA}^{\cU}(\cal M_i, \Phi_i),\Phi^\cU\right)
\]
to be the projection onto the subspace corresponding to the image of $W$.

We would like to show that $p_{\Phi}\left( \prod^{\cU}_{GR} (\cal M_i, \cal H_{\Phi_i})\right)p_{\Phi} = \prod_{\HA}^{\cU}(\cal M_i, \Phi_i)$.  In other words, the Hilbert algebra ultraproduct is the corner of the Groh-Raynaud ultraproduct corresponding to $p_{\Phi}$.

\begin{thm}\label{ultracorner}
    $p_{\Phi}\left( \prod^{\cU}_{GR} (\cal M_i, \cal H_{\Phi_i})\right)p_{\Phi} = \R_{\lambda}\left(\prod_{\HA}^{\cU}(\cal M_i, \Phi_i)\right)$.
\end{thm}

\begin{proof}
    By \cite[Theorem 3.18]{AH}, $\prod^{\cU}_{GR} (\cal M_i, \cal H_{\Phi_i})$ acts standardly on $\prod^{\cU} (\cal H_{\Phi_i})$ with antilinear isometry $J_{\cU} := (J_i)^{\cU}$.  Consider the projection $q = p_{\Phi}J_{\cU}p_{\Phi}J_{\cU}$.  It is clear that $p_{\Phi} \in \prod^{\cU}_{GR} (\cal M_i, \cal H_{\Phi_i})$ since it is a projection with image fixed by every element of $\left(\prod^{\cU}_{GR} (\cal M_i, \cal H_{\Phi_i})\right)' = \prod^{\cU}_{GR} (\cal M_i', \cal H_{\Phi_i})$ (where this equality is given by \cite[Theorem 3.22]{AH}).  Then by \cite[Lemma 2.6]{Ha75} we have that $p_{\Phi}\left( \prod^{\cU}_{GR} (\cal M_i, \cal H_{\Phi_i})\right)p_{\Phi} \cong q\left(\prod^{\cU}_{GR} (\cal M_i, \cal H_{\Phi_i})\right)q$ where the latter acts standardly on the subspace 
    \[
    q\left(\prod^{\cU}( \cal H_{\Phi_i})\right) \cong p_{\Phi}\left( \prod^{\cU} (\cal H_{\Phi_i})\right) \cong L^2\left(\prod_{\HA}^{\cU}(\cal M_i, \Phi_i),\Phi^\cU\right).
    \]
    Moreover, since $J_i$ acts as an isometry on $\cal H_{\Phi_i}$ and conjugation by $J_i$ acts as an isometry from $(\cal M_i, \cal H_{\Phi_i})$ to $(\cal M_i, \cal H_{\Phi_i})'$, we see that the restriction of $J_{\cU}$ to the image of $p_{\Phi}$ agrees with the modular conjugation in $\prod_{\HA}^{\cU}(\cal M_i, \Phi_i)$.  Thus $p_{\Phi}$ commutes with $J_{\cU}$ and therefore $q = p_{\Phi}$.  We also have, by Theorem \ref{correctness} and Proposition \ref{ultrapres}, that $\R_{\lambda}\left(\prod_{\HA}^{\cU}(\cal M_i, \Phi_i)\right)$ acts standardly on $L^2\left(\prod_{\HA}^{\cU}(\cal M_i, \Phi_i),\Phi^\cU\right)$.  If $(v_i)^\bullet \in \eta_{\Phi^{\cU}}\left(\prod_{\HA}^{\cU}(\cal M_i, \Phi_i)\right)$, then $\pi((v_i)^\bullet)$ is obviously in the Groh-Raynaud ultraproduct.  Furthermore, it follows from a simple calculation that $\pi((v_i)^\bullet)$ commutes with $p_{\Phi}$.  Now the result follows by the fact that left Hilbert algebras give rise to unique standard forms up to unitary equivalence.
\end{proof}

The proof of the next theorem is nearly identical to those given in \cite[Proposition 4.11]{AH} and \cite[Theorem 1.5]{MT}.

\begin{thm}\label{ultracontinuous1}
    Let $(x_i) \in \ell_{\Phi}^\infty(\mathfrak{A}_i)$.  Then the following statements are equivalent:
    \begin{enumerate}
        \item $(x_i) \in \cal N_{\cU} := \{(x_i) \in \ell_{\Phi}^\infty(\mathfrak{A}_i) \ : \ (x_i) \cal I_{\cU} \subseteq \cal I_{\cU} \text{ and } \cal I_{\cU} (x_i)  \subseteq \cal I_{\cU}\}$.
        \item For any $\epsilon > 0$, there exists $a > 0$ and $(y_i) \in \ell_{\Phi}^\infty(\mathfrak{A}_i)$ such that
        \begin{itemize}
            \item $\lim_{i \to \cU} \|x_i - y_i\|^\#_{\Phi_i} < \epsilon$; and
            \item $y_i \in M(\sigma^{\Phi_i}, [-a, a])$ for all $i \in I$.
        \end{itemize}
        \item $(x_i)^\bullet$ is a $\sigma$-continuous vector.
    \end{enumerate}
\end{thm}

This tells us that elements in the domain of definition are in the Hilbert space ultraproduct if and only if they are $\sigma$-continuous vectors.  In the \wstar-probability space setting, we would be done here since every element of the Ocneanu ultraproduct is given by a vector in the corresponding Hilbert space.  However, in our more general setting, we still need to deal with the strong closure.

\begin{cor}
    An element $(x_i)^\bullet \in \ell^{\infty}_{\Phi}$ is in $\cal N_{\cU}$ if and only if $(\Delta_{\Phi_i}^{it}x_i)^\bullet \in \cal N_{\cU}$ for all $t \in \bbR$.
\end{cor}

\begin{proof}
    First, note by translation invariance of continuous group homomorphisms, $(\sigma_t^{\Phi_i}(x_i))^\bullet$ is continuous at $t = 0$ if and only if it is continuous at all $t \in \bbR$.  Thus since $\sigma_t^{\Phi_i}(x_i) = \pi(\Delta^{it}x_i)$, we conclude.
\end{proof}

We are now ready to prove the final piece of our main result, handling those $\sigma$-continuous points that are not represented by vectors.

\begin{thm}\label{ultracontinuous2}
    Let $x = (x_i)_{\cU}$ be an element of the Groh-Raynaud ultraproduct $\prod^{\cU}_{GR} (\cal M_i, \cal H_{\Phi_i})$ where $x_i \in \cal M_i$ for all $i$.  The following are equivalent.
    \begin{enumerate}
        \item $p_{\Phi} (x_i)^\bullet = (x_i)^\bullet p_{\Phi}$ so that $x$ is an element of the Hilbert algebra ultraproduct.
        \item $(x_i)^\bullet$ is a strong$^*$-continuous point of $(\sigma^{\Phi_i}_t)_{\cU}$ so that $x$ is an element of the generalized Ocneanu ultraproduct.
    \end{enumerate}
\end{thm}

\begin{proof}
    The (1) $\implies$ (2) direction follows from Theorem \ref{ultramod} and the fact that the modular group is known to be strong$^*$-continuous.  In this case, the ultraproduct of the modular automorphism groups acts as the modular automorphism group of the weight induced on the Hilbert algebra ultraproduct.  The latter agrees with the ultraproduct weight on the Groh-Raynaud ultraproduct since $p_{\Phi}$ is a projection onto the image of the isometry $W$.

    Now we prove (2) $\implies$ (1).  Assume $(x_i)^\bullet$ is a strong$^*$-continuous point of $(\sigma^{\Phi_i}_t)_{\cU}$.  Then for all $(v_i)^\bullet \in \cal N_{\cU}$, we have $(\sigma_t^{\Phi_i}(x_i)v_i)^\bullet$ strong$^*$ converges to $(x_i v_i)^\bullet$ as $t \to 0$.  So we compute
    \begin{align}
        (\sigma_t^{\Phi_i}(x_i)v_i)^\bullet &= (\Delta_{\Phi_i}^{it} x_i \Delta_{\Phi_i}^{-it} v_i)^\bullet \nonumber \\
        &= (\Delta_{\Phi_i}^{it} x_i (\Delta_{\Phi_i}^{-it} v_i))^\bullet. \nonumber
    \end{align}
    We also have that $(v_i)^\bullet \in \cal N_{\cU}$ if and only if $(\Delta^{-it}v_i)^\bullet \in \cal N_{\cU}$ by Theorem \ref{ultracontinuous1}.  So we have $(\Delta^{it}v_i)^\bullet \in \cal N_{\cU}$.  By a reparameterization we can replace $(\Delta^{it}v_i)^\bullet$ with $v_i$.
    
    Thus $(\Delta^{it}x_iv_i)^\bullet$ strong$^*$ converges to $(x_i v_i)^\bullet$ as $t \to 0$ for all $(v_i)^\bullet \in \cal N_{\cU}$.  Applying Theorem \ref{ultracontinuous1} again, we conclude $(x_i)^\bullet(v_i)^\bullet \in \cal N_{\cU}$ and $(x_i^*)^\bullet(v_i)^\bullet \in \cal N_{\cU}$ for all $(v_i)^\bullet \in \cal N_{\cU}$.  This is equivalent to $p_{\Phi}(x_i)^\bullet = (x_i)^\bullet p_{\Phi}$.  Thus, we are done.
\end{proof}

\begin{cor}\label{GOisCornerOfGR}
    The generalized Ocneanu ultraproduct is a von Neumann subalgebra of the Groh-Raynaud ultraproduct that is the image of a faithful normal conditional expectation.
\end{cor}

\begin{proof}
    By Theorem \ref{ultracontinuous2} and Proposition \ref{correctness}, the generalized Ocneanu ultraproduct is a von Neumann subalgebra.

    By Theorem \ref{ultramod} and since the ultraproduct weight is the restriction of the ultraproduct weight on the Groh-Raynaud ultraproduct, the modular automorphism group on the generalized Ocneanu ultraproduct agrees with the modular automorphism group on the Groh-Raynaud ultraproduct.  Since continuous elements are invariant under $\sigma^{\Phi^{\cU}}_t$, we have the desired conditional expectation by Takesaki's Theorem.
\end{proof}

Altogether,

\begin{thm}
    For any family $(\cal M_i, \Phi_i)$ of weighted von Neumann algebras, the generalized Ocneanu ultraproduct acting on $p_{\Phi}(\prod^{\cU} \cal H_{\Phi_i})$ is spatially isomorphic to the Hilbert algebra ultraproduct of the corresponding family of full Hilbert algebras $(\mathfrak{A}_i, \cal H_{\Phi_i})$.  By transitivity, both of these are isomorphic to the model-theoretic ultraproduct in $\Md(T_{\vNa})$.
\end{thm}

\section{Applications to Computability and Ultraproduct Embeddings}\label{SectionUniversalTheories}

Having now developed all of the necessary technology to do so, we use this section to study the (un)decidability of the universal theories of weighted von Neumann algebras with unbounded weights.  We expand our catalogue of von Neumann algebras with undecidable universal theory.  Fittingly, the first entry with unbounded weight is the infinite analogue of the hyperfinite II$_1$ factor.  Our model-theoretic work facilitates results of this kind, which were not possible in earlier formalisms.

\begin{thm} \label{ThmUndecIIinfty}
    The universal theory of the hyperfinite II$_\infty$ factor $\R_{0,1}$ is not computable in the language $\cal L_{\vNa}$.
\end{thm}

\begin{proof}
    Let $\sup_{x}\theta(x)$ be a universal sentence in the language of tracial von Neumann algebras where again we assume all variables range over the unit ball.  Let $P_1$ denote the set of trace $1$ projections in $\R_{0,1}$.  It can be seen that $P_1$ is a quantifier-free definable set by the formula
    \[
    \max \{d(p^2,p), d(p^*,p), |\Phi(p) - 1|\}.
    \]
    Note that for every $p \in P_1$, the set
    \[
    p\R_{0,1}p := \{pxp \ : \ x \in \R_{0,1} \}
    \]
    is a copy of $\R$ and the restriction of the $\Phi$ on $\R_{0,1}$ to this set is the canonical trace on $\R$.  Furthermore, the map $x \mapsto pxp$ sends the unit ball of $\R_{0,1}$ to the unit ball of the corresponding copy of $\R$.  Thus, computing the $\sup_{x} \theta(x)$ in any one copy of $\R$ is the same as computing it in any other.  So consider the universal sentence in $\cal L_{\vNa}$ given by
    \[
    \psi = \sup_{p \in P_1} \sup_{x} \theta(pxp).
    \]
    By assumption, $\psi$ is computable in $\R_{0,1}$.  But
    \[
    \psi^{\R_{0,1}} = (\sup_{x} \theta(x))^{\R},
    \]
    contradicting the uncomputability of the universal theory of $\R$.
\end{proof}

The proof of the above can be made more general to prove the following principle.

\begin{thm}
    Let $(\cal M, \tau_0)$ be a II$_1$ factor equipped with its canonical trace.  Assume $(\cal M, \tau_0)$ has undecidable universal theory.  If $\cal N = \cal M \otimes B(\cal H)$ and the tracial weight $\tau$ is the II$_\infty$ amplification of the trace $\tau_0$, then $(\cal N, \tau)$ also has undecidable universal theory.
\end{thm}

Theorem \ref{ThmUndecIIinfty} tells us that the universal theory of $\R_{0,1}$ can compute the Turing degree $\mathbf{0}'$.  Using the techniques and results from \cite{GHHyp}, we will show that, conversely, the universal theory of $\R_{0,1}$ can be computed from $\mathbf{0}'$. We will do this by demonstrating a computable presentation of $\R_{0,1}$ (see \cite[Remark 3.12]{GHHyp}).

\begin{prop}\label{PropCompPres}
    $\R_{0,1}$ admits a computable presentation.
\end{prop}

\begin{proof}
    Take an effective enumeration $A_k$ of a dense subset of the unit ball of the hyperfinite II$_1$ factor $\cal R$ as in \cite[Remark 3.11]{GHHyp}.  Consider the matrix units of $B(\cal H)$, namely, $e_{i,j}$ for $i, j \in \bbN$.  By a simple zigzag argument, there is a computable enumeration $d_{m} := e_{i,j}$ where $f: \bbN^2 \to \bbN$ is a computable bijection and $m = f(i,j)$.  Using similar zigzag arguments, we can enumerate finite sums of the form $A_{k_1} \otimes d_{m_1} + \ldots + A_{k_N} \otimes d_{m_N}$.  Let $B_{n}$ be such an enumeration.  It is now clear that $B_{n}$ is an effective enumeration of the unit ball of $\R_{0,1}$.  Furthermore, we can effectively compute $\|B_n\|_{\tau}$ by computing the $2$-norms of the left tensors of the diagonal entries in $\R$ and adding them up.
\end{proof}

Synthesizing Theorem \ref{ThmUndecIIinfty} and Proposition \ref{PropCompPres}, we arrive at the following corollary.

\begin{cor}
    The universal theory of $\R_{0,1}$ has Turing degree $\mathbf{0}'$.
\end{cor}

By the same argument as above, we have

\begin{thm}
    Assume $\cal M_0$ is a II$_1$ factor with a computable presentation, then $\cal M := \cal M_0 \otimes B(\cal H)$ with the associated tracial weight has a computable presentation.
\end{thm}

We also have the following corollary by \cite[Theorem 5.2]{GH2}.

\begin{cor}
    $\R_{0,1}$EP has a negative solution. In other words, there is no effectively enumerable subset $T$ of the full theory of $\R_{0,1}$ such that, for any $\cal L_{\vNa}$-structure $\cal M$, if $\cal M \vDash T$ , then $\cal M$ embeds into an ultrapower of $\R_{0,1}$.
\end{cor}

Since $\R_{0,1}$ embeds in any $II_{\infty}$ factor, we may argue as in \cite[Remark 5.4]{GH2} to conclude:

\begin{cor}
    There is no effectively enumerable theory $T$ extending the theory of $II_{\infty}$ factors with the property that every model of $T$ embeds into an ultrapower of $\R_{0,1}$.
\end{cor}

Arguing as in \cite[Corollary 5.5]{GH2} we also have the following.

\begin{cor}
    There is a sequence $\cal M_1, \cal M_2, \ldots$ of separable II$_\infty$ factors, none of which embed into an ultrapower of $\R_{0,1}$, and such that, for all $i < j$, $\cal M_i$ does not embed into an ultrapower of $\cal M_j$.
\end{cor}

Notice that when $(\cal M, \Phi)$ is a weighted von Neumann algebra, the restriction of $\Phi$ to the centralizer $\cal M_{\Phi}$ is always a faithful normal tracial weight.  In \cite{Ha77}, Haagerup shows that the restriction of $\Phi$ to the centralizer is not always semifinite by providing an example $(\cal M, \Phi)$ where $\cal M_{\Phi}$ is not a semifinite von Neumann algebra. Recall that a von Neumann algebra is said to be semifinite if it admits a faithful normal semifinite tracial weight.

\begin{defn}
    A faithful normal semifinite weight $\Phi$ on $\cal M$ is said to be \textbf{strictly semifinite} if the restriction of $\Phi$ to $\cal M_{\Phi}$ is a semifinite weight.
\end{defn}

\begin{lem}{\cite[Lemma 2.3]{HS}}
    Let $\Phi$ be a faithful normal semifinite lacunary weight on $\cal M$. Then $\Phi$ is strictly semifinite. 
\end{lem}

\begin{lem}\label{CentralizerDef}
    If $(\cal M, \Phi)$ is a weighted von Neumann algebra and $\Phi$ is lacunary, then the projection of $\cal H_{\Phi}$ onto $\eta_{\Phi}(\cal M_{\Phi})$ is a computably definable function.
\end{lem}

\begin{proof}
    Recall that $\Delta_{\Phi} = R_{\Phi}^{-1}(2- R_{\Phi})$ and that the spectra of both $R_{\Phi}$ and $(2 - R_{\Phi})$ are both contained in $[0,2]$. By functional calculus, we have that $\Delta$ is $\lambda$-lacunary if and only if $\spec(R_{\Phi}) \cap (2 - \frac{2}{\lambda + 1}, \frac{2}{\lambda + 1}) = \{1\}$. Suppose $f_{\lambda}(x)$ is a computable function such that $f_{\lambda}(1) = 1$ and $f_{\lambda}$ is supported on $(2 - \frac{2}{\lambda + 1}, \frac{2}{\lambda + 1})$. For concreteness, one can choose $f_{\lambda}$ to be a suitable horizontal scaling of $g(x)$ where 
    \[
    g(x) = \begin{cases} 
    \exp \left(-\frac{1}{1-(1-x)^2} + 1 \right) & \text{if } x \in (0,2) \\
    0 & \text{otherwise }
    \end{cases}.
    \]
    It follows by functional calculus that $f_{\lambda}(R_{\Phi})$ is equal to the projection onto the $\lambda = 1$ eigenspace of $R_{\Phi}$, namely the centralizer. By Takesaki's theorem, this projection is well-defined onto the centralizer.
\end{proof}

Now we prove a generalization of \cite[Proposition 5.3]{AGH}.

\begin{thm}
    Let $(\cal M, \Phi)$ be a weighted von Neumann algebra such that $\Phi$ is lacunary.  Then $(\cal M_{\Phi})^{\cU} \cong  \cal (M^{\cU})_{\Phi_{\cU}}$.
\end{thm}

\begin{proof}
    First we prove $(\cal M_{\Phi})^{\cU} \subseteq \cal (\cal M^{\cU})_{\Phi_{\cU}}$.  By Takesaki's theorem, there is a conditional expectation $E_{\Phi} : \cal M \to \cal M_{\Phi}$, which is furthermore implemented by a projection on the underlying Hilbert spaces $P_{\Phi} : \cal H_{\cal M} \to \cal H_{\cal M_{\Phi}}$.  Therefore we have a natural identification $\ell^{\infty}_{\Phi}(\cal M_{\Phi}) \hookrightarrow \ell^{\infty}_{\Phi}(\cal M)$.  So consider $x = (x_i) \in (\cal M_{\Phi})^{\cU}$.  Then by Theorem \ref{ultramod}, for all $t \in \bbR$, we have
    \[
    \sigma_t^{\Phi_\cU}(x) = (\sigma_t^{\Phi_i}(x_i))^\bullet = (x_i)^\bullet = x
    \]
    whence the claim follows by taking strong closures.

    Now we want to prove $\cal (\cal M^{\cU})_{\Phi_{\cU}} \subseteq (\cal M_{\Phi})^{\cU}$.  Arguing as in the previous lemma, the lacunary assumption implies the projection from $\cal H_{\cal M}$ onto the closure of $\eta_{\Phi}(\cal M_{\Phi})$ is definable.  Therefore we have $\eta_{\Phi_{\cU}}(\cal (\cal M^{\cU})_{\Phi_{\cU}}) \subseteq \eta_{\Phi_{\cU}}((\cal M_{\Phi})^{\cU})$.  Since $1 \in \cal M_{\Phi}$, we have that $\cal M_{\Phi}$ has the same identity as $\cal M$.  We also assumed that $\Phi$ is lacunary and therefore is semifinite on $\cal M_{\Phi}$, and normality and faithfulness on $\cal M_{\Phi}$ is automatic.  Thus by equivalence of faithful normal semifinite weights and full left Hilbert algebras, by taking strong closures, we are done.
\end{proof}


We conclude with a series of questions and conjectures.  Theorem \ref{CentralizerDef} together with Theorem \ref{ThmUndecIIinfty} immediately raise the following question.

\begin{question}
    For which naturally occurring weighted von Neumann algebras $(\cal M, \Phi)$ do we have $\Phi$ is lacunary, $\cal M_{\Phi}$ is isomorphic to the hyperfinite II$_\infty$ factor and the restriction of $\Phi$ to $\cal M_{\Phi}$ is the canonical tracial weight on it?
\end{question}

At present, little is known about the computability of universal theories of hyperfinite type III$_0$ factors.  One reason for this is the fact (\cite[Theorem 6.13]{AH}) that the class of type III$_0$ von Neumann algebras are not axiomatizable and neither (\cite[Theorem 6.18]{AH}) is the class of III$_0$ factors.  On the other hand, we direct the reader to \cite[Section 5]{Dab}, wherein Dabrowski shows the class of III$_0$ factors is an uncountable union of elementary classes by axiomatizing discrete crossed product decompositions with given parameters.  The reader should also consult \cite{GoldbringHoudayer} for results on existential closures of III$_0$ factors.

We record the following from the discussion following \cite[Lemma 4.1]{CW}.

\begin{thm}\label{centralizerIII0}
    Let $\cal M$ be a type III$_0$ factor and let $\Phi$ be a faithful lacunary weight of infinite multiplicity on $\cal M$.  Then
    \[
    \cal M_{\Phi} = \int_{B}^{\oplus} \cal M(b) \mathrm{d}v_{B}(b)
    \]
    where each $\cal M(b)$ is a type II$_\infty$ factor, and the center $C_{\Phi}$ of the centralizer $\cal M_{\Phi}$ is isomorphic to $L^\infty(B, v_B)$. If $\cal M$ is injective, then $\cal M(b)$ is the hyperfinite II$_\infty$ factor.
\end{thm}

In fact, by \cite[Lemma 5.3.2]{Connes73}, we have:

\begin{thm}\label{III0Lacunary}
    Every type III$_0$ factor admits admits a faithful normal semifinite lacunary weight of infinite multiplicity.
\end{thm}

It therefore is desirable to understand the theory of a direct integral of tracial weighted von Neumann algebras.  For tracial von Neumann algebras, several results about elementary equivalence of direct integrals are given by Farah-Ghasemi in \cite{FG} and Gao-Jekel in \cite{GaoJekel}.  Notably, in the discussion after \cite[Theorem D]{GaoJekel}, it is suspected extending their results to II$_\infty$ may be difficult due to the lack of axiomatizability of this class.  On the other hand, if we assume that the the weights involved are tracial, this class is, in fact, axiomatizable in our language.

In both \cite{FG} and \cite{GaoJekel}, a key point is that the center of a tracial von Neumann algebra is definable.  By a proof similar to \cite[Corollary 4.3]{FHS1}, we have

\begin{lem}
    If $(\cal M, \Phi)$ is a tracial weighted von Neumann algebra, then the center of $\cal M$ is definable.
\end{lem}

It should thus be possible to prove analogous results about tracial weighted von Neumann algebras. We conjecture a positive answer to the following questions.

\begin{question}
    Assume $(\cal M, \Phi)$ is a tracial weighted von Neumann algebra such that
    \[
    \cal M = \int_{B}^{\oplus} \cal M(b) \mathrm{d}v_{B}(b)
    \]
    where each $\cal M(b)$ is a type II$_\infty$ factor and $\cal M(b_1) \cong \cal M(b_2)$ for all $b_1, b_2 \in B$.  Does the computability of the universal theory of $(\cal M, \Phi)$ imply the computability of the universal theory of $(\cal M(b), \tau_b)$ for all $b \in B$?  What if we instead only assume $\cal M(b_1) \equiv \cal M(b_2)$ for all $b_1, b_2 \in B$?
\end{question}

A positive answer to either of these questions would imply that every hyperfinite III$_0$ factor $\cal M$ admits a weight such that the universal theory of $(\cal M, \Phi)$ is undecidable.  Together with the results in \cite{GH2}, \cite{AGH} and \cite{Connes76}, this would imply that all hyperfinite factors admit such a weight.

\begin{question}
    What other examples of weighted von Neumann algebras with unbounded weights can we find with undecidable universal theory?
\end{question}

Finally,

\begin{question}
    What new insights can the theory developed here reveal about more specific settings? For example, if $(\cal M, \Phi)$ is a II$_1$ factor with an unbounded weight, what if anything, can the theory of $(\cal M, \Phi)$ tell us about the theory of $(\cal M, \tau)$?
\end{question}

\section*{Acknowledgements} We would like to thank Isaac Goldbring, Bradd Hart and Thomas Sinclair for many discussions that led to the ideas in this paper.  We also thank David Sherman for many comments on an early draft of the author's PhD thesis, from which many of the results in this paper are drawn, and for asking or inspiring many of the questions.


\end{document}